\documentclass[reqno]{amsart}
\usepackage{amssymb,amsthm,amsfonts,amstext}
\usepackage{amsmath}
\usepackage[ansinew]{inputenc} 
\usepackage{newcent}       
\usepackage{helvet}         
\usepackage{courier}        
\makeatletter \@addtoreset{equation}{section} \makeatother
\usepackage{color}

\renewcommand\thefigure{\thesection.\@arabic\c@figure}
\renewcommand\thetable{\thesection.\@arabic\c@table}
\newtheorem{theorem}{Theorem}[section]

\newtheorem{proposition}[theorem]{Proposition}

\newtheorem{remark}[theorem]{Remark}

\newcommand{\mc}[1]{{\mathcal #1}}

\newcommand{\bb}[1]{{\mathbb #1}}
\newcommand{\<}{\langle}
\renewcommand{\>}{\rangle}
\newcommand{\Z}{\mathbb Z}
\newcommand{\E}{\mathbb E}
\newcommand{\R}{\mathbb R}
\newcommand{\N}{\mathbb N}
\renewcommand{\S}{\mathbb S}
\newcommand{\Y}{\mathcal{Y}}
\newcommand{\M}{\mathcal{M}}
\newcommand{\I}{\mathcal{I}}
\newcommand{\B}{\mathcal{B}}
\newcommand{\D}{\mathcal{D}}

\newcommand{\A}{\mathcal{A}}

\renewcommand{\P}{\mathbb P}

\newcommand{\W}{\mathcal{W}}
\newcommand{\K}{\mathcal{K}}
\newcommand{\HW}{\widetilde{H}}
\renewcommand{\d}{d}
\newcommand{\e}{\mathfrak{d}}

\begin{document}

\author{Sunder Sethuraman}
\address{Mathematics, University of Arizona, 617 N. Santa Rita Ave., Tucson, AZ 85721.} \email{sethuram@math.arizona.edu}

\title [Microscopic derivation of a fractional stochastic Burgers equation]{On microscopic derivation of a fractional stochastic Burgers equation}


\begin{abstract}
We derive from a class of microscopic asymmetric interacting particle systems on $\Z$, with long range jump rates of order $|\cdot|^{-(1+\alpha)}$ for $0<\alpha<2$, different continuum fractional SPDEs.  More specifically, we show the equilibrium fluctuations of the hydrodynamics mass density field of zero-range processes, depending on the stucture of the asymmetry, and whether the field is translated with process characteristics velocity, is governed in various senses by types of fractional stochastic heat or Burgers equations.

The main result:  Suppose the jump rate is such that its symmetrization is long range but its (weak) asymmetry is nearest-neighbor.   Then, when $\alpha<3/2$, the fluctuation field in space-time scale $1/\alpha:1$, translated with process characteristic velocity, irrespective of the strength of the asymmetry, converges to a fractional stochastic heat equation, the limit also for the symmetric process.  However, when $\alpha\geq 3/2$ and the strength of the weak asymmetry is tuned in scale $1-3/2\alpha$, the associated limit points satisfy a martingale formulation of a fractional stochastic Burgers equation.

\end{abstract}

\maketitle

\section{Introduction}
\label{intro_section}

The purpose of this paper is to derive from a class of microscopic zero-range interacting particle systems on $\Z$, with asymmetric long range jump rates, certain continuum `fractional Burgers' and other stochastic partial differential equations (SPDE).  Our motivations are three fold:  

First, these results will be seen to complement recent work and conjectures in \cite{bgs} which infer certain `long range' KPZ class variance orders from the study of occupation times in asymmetric exclusion processes on $\Z$ with long range jump rates of order $|\cdot|^{-(1+\alpha)}$ for $\alpha>0$.

Second, given the interest in anomalous scales and previous work on deterministic fractional Burgers equations \cite{BFW}, \cite{KW}, \cite{Vazquez}, \cite{W_book}, it is a natural problem to try to understand the corresponding SPDEs.  

Third, although with respect to nearest-neighbor interacting systems on $\Z$, there has been much interest in KPZ Burgers equation which has been interpreted and understood in several ways (cf. \cite{ACQ}, \cite{Assing}, \cite{BG}, \cite{Funaki}, \cite{GJ}, \cite{gjs}, \cite{H}, \cite{Sasamoto-Spohn} and references therein), there seems to be little work in deriving such equations in the long range setting.  

We now expand on these motivations before discussing results.  

\subsection{Occupation times and KPZ class exponents}
\label{occ_subsection}
Consider the exclusion process on $\Z$ with single particle jump probability $p(x,y) = p(y-x)$.
In such a process, each particle jumps at rate $1$ and displaces according to $p$, except in that jumps to already occupied vertices are suppressed.  The configuration $\eta_t$ at time $t\geq 0$ is a collection of occupation numbers $\eta_t = \{\eta_t(x): x\in \Z\}$ where $\eta_t(x)$ is particle number at $x$ at time $t$.  The system is a Markov process with a family of invariant measures $\nu_\rho$, each indexed on configurations with asymptotic density $\rho\in [0,1]$; in fact, $\nu_\rho$ is a product of Bernoulli measures over the lattice points in $\Z^d$ \cite{Liggett}.

Suppose now $\rho = 1/2$ and the system is begun with distribution $\nu_{1/2}$.  It is known that the variance of the occupation time at the origin satisfies
$${\rm Var}\int_0^t \eta_s(x) - 1/2 ds \ \sim \ 2t\int_0^t \bar\P_{\nu_{1/2}}\big(R_s = 0\big)ds,$$
where $R_s$ is the position of a `second-class' particle initially at the origin.  Such a particle moves as a regular particle but also must exchange places when other regular particles decide to displace to its location.  

When $p$ is finite-range and with a drift $\sum_x xp(x)\neq 0$, it is known (cf. \cite{Balazs-Seppalainen1}, \cite{QV}) that
${\rm Var} R_t = O(t^{4/3})$.  Such second-class particle variances are known to connect to the variance of the height function for certain interfaces formed from the particle configuration \cite{Baik}.
Now, with a Gaussian ansatz, one posits decay $\P_{\nu_{1/2}}\big(R_s = 0) \sim \big({\rm Var} R_s\big)^{-1/2}$, which in particular would give ${\rm Var} \int_0^t \eta_s(0) - 1/2 ds \sim t^{4/3}$.  Although this type of local variance estimate has not been proved, superdiffusive lower bounds have been shown \cite{B}.  For the purposes of this article, we say KPZ class variance orders are those of the second-class particle (or the occupation time), as the correspondence with a height function is not obvious in the long range setting.  

Now, when $p$ is long range, that is $s(x) = \big(p(x) + p(-x)\big)/2 = O(|x|^{-(1+\alpha)})$ and $a(x) = \big(p(x) -p(-x)\big)/2$ is say supported on nearest-neighbor steps $a(\pm 1)\neq 0$, one can ask about the occupation time variance orders.  Surprisingly, in \cite{bgs}, it was shown, for several types of asymmetric jump probabilities including $p$, when $\alpha=3/2$, the variance is of order $O(t^{4/3})$.  Also for $\alpha>2$, when the jump law has more than $2$ moments, it was proved the variance is of the {\it same} order as that for the finite-range system with a jump probability with the same drift.  Then, it was conjectured (cf. Conjecture 2.17 in \cite{bgs}), given that the system is more volatile as $\alpha$ grows, that the variance should be of order $O(t^{4/3})$ for all $\alpha\geq 3/2$, a `long range' extension of the standard KPZ class variance orders.  

When $0<\alpha<3/2$, as shown in \cite{bgs} the variance has the same order as for the corresponding symmetric process with symmetrized jump probability $s$, which was computed to be $O(t^{2-1/\alpha})$ for $\alpha\geq 1$ and $O(t)$ for $0<\alpha<1$.  Therefore, in a sense, the long range KPZ class variance orders should match those of the finite-range class when $\alpha\geq 3/2$, and those of the symmetrized system when $\alpha<3/2$.  

These are in a sense `local' fluctuation results.  One can whether also the long range `bulk' fluctuations, that is those of the empirical density field, also follow such $\alpha$-dependent characterizations.  Given that the computations in \cite{bgs} were performed for the exclusion process, one can ask also whether the phenomenon extends to other mass-conservative interacting particle systems.

\subsection{Hydrodynamics and deterministic fractional Burgers equations}
\label{hyd_subsection}
For an array of weakly asymmetric nearest-neighbor exclusion processes on $\Z$, with jump probability $p(1) = 1/2 + c/n^{1/2}$ and $p(-1) = 1/2 - c/n^{1/2}$, it is well known that the diffusively scaled empirical density field, 
$$
\frac{1}{n}\sum_x \eta_{nt}(x) \delta_{x/n^{1/2}},$$
 when started from an initial measure with density profile $\sim \rho_0(x)$ and appropriate bounded entropy, converges weakly to the unique solution of the hydrodynamic equation
\begin{equation}
\label{nn_hyd}
\partial_t \rho \ = \ \frac{1}{2}\Delta \rho -2c\nabla \rho(1-\rho); \ \ \rho(0,x) = \rho_0(x).\end{equation}
See \cite{KL} for a complete account.

However, when $1\leq \alpha<2$, for long range weakly asymmetric processes, that is when $s(x) = O(|x|^{-(1+\alpha)})$ and $a$ is nearest-neighbor, $a(1) = c/n^{1-1/\alpha}$ and $a(-1) = -c/n^{1-1/\alpha}$, the long-range density field $(1/n^{1/\alpha}) \sum_x \eta_{nt}(x) \delta_{x/n^{1/\alpha}}$ formally converges to the solution of
$$\partial_t \rho \ = \ \Delta^{\alpha/2}\rho -2c\nabla \rho(1-\rho); \ \ \rho(0,x) = \rho_0(x).$$
Here $\Delta^{\alpha/2}$ is the fractional Laplacian, 
$$\Delta^{\alpha/2}G(x) = \frac{1}{2}\int_\R s(y) \left\{ G(x+y) - 2G(x) + G(x-y)\right\}dy.$$

When $0<\alpha<1$, no matter the order of the asymmetry $a(\pm 1)$, the density field converges to the solution of
$$\partial_t \rho \ = \ \Delta^{\alpha/2}\rho; \ \ \rho(0,x) = \rho_0(x).$$

When $\alpha>2$, under diffusive scaling and $a(\pm 1) = \pm cn^{-1/2}$, the density field tends to the solution of \eqref{nn_hyd}.  Also, when $\alpha = 2$, under `log' adjusted $a(\pm 1)= \pm c\log(n)/[(n\log n)^{1/2}]$ the field $(n\log n)^{-1/2}\sum_x \eta_{nt}(x)\delta_{x/(n\log n)^{1/2}}$ converges to the solution of \eqref{nn_hyd}. 

For different particle systems, such as zero-range processes (cf. Section \ref{notation-results}), which also have a family of invariant measures $\nu_\rho$ indexed by density, the formal long range hydrodynamic equations take form
\begin{equation}
\label{zr_hyd}
 \partial_t \rho \ = \ \Delta^{\alpha/2}\tilde{g}(\rho) -2c\nabla \tilde{g}(\rho)
\end{equation}
when $1\leq \alpha<2$ and
$\partial_t \rho = \Delta^{\alpha/2}\tilde{g}(\rho)$ when $0<\alpha<1$, and $\partial_t\rho = (1/2)\Delta\tilde{g}(\rho) -2c\nabla\tilde{g}(\rho)$ in diffusive and `log'-adjusted scales when $\alpha\geq 2$, in terms of a (nonlinear) `flux' function $\tilde{g}$.  

 See \cite{Jara_tagged_arxiv} in this context which addresses hydrodynamics, and also \cite{BFW},  \cite{Vazquez}, \cite{W_book} which consider uniqueness and regularity of related equations. 

It is natural to ask about the equilibrium fluctuations corresponding to these hydrodynamic limits.  In particular, starting from an invariant measure $\nu_\rho$, what are the limits of the fluctuation field $(1/n^{1/2\alpha})\sum_x \big(\eta_{nt}(x) -\rho\big) \delta_{x/n^{1/\alpha}}$, when say $0<\alpha<2$?  When the process is symmetric, that is $p=s$, such limits were considered in \cite{Jara_longrange} (cf. Proposition \ref{symmetric_prop}).  The general answer, well understood from a perturbative view and in many finite range examples, is that the fluctuation limit should be a linearization of the hydrodynamic equation, forced with a certain White noise \cite{BR}, \cite{Dittrich}, \cite{Ravishankar}, \cite{Spohn}.

\subsection{KPZ and stochastic Burgers equations}
\label{KPZ_subsection}
The KPZ equation, 
$$\partial_t h \ = \ a \Delta h + b \big(\nabla h\big)^2 + c \dot\W_t,$$
has stimulated much recent activity in the probability/math physics literature \cite{Corwin_review}.  Here, $h(t,x)$ represents the continuum height of certain interfaces with certain growth rules.  Part of the equation's mystique is that it is ill posed:  The noise is not regular enough to allow a strong solution, and the square nonlinearity prevents a weak formulation.  

Nevertheless, formally, the Cole-Hopf transform $Z(t,x) = e^{\lambda h(t,x)}$ with $\lambda = a/b$ satisfies the linear stochastic heat equation $\partial_t Z = a \Delta Z + (ac/b)Z \dot\W_t$ which is well-defined \cite{Walsh}.  One then declares $\log Z(t,x)$ as the `solution' to the KPZ equation.  In a recent tour-de-force \cite{H}, what actual equation $\log Z(t,x)$ satisfies and its relation to the KPZ equation was made precise.

From the microscopic point of view, the microscopic height function satisfies $h(t,x)-h(t,x+1) = \eta_t(x)$ where as before $\eta_t(x)$ is the particle number at $x$ at time $t$.  In nearest-neighbor exclusion processes, starting from $\nu_\rho$, with jump probability which is weakly asymmetric in that $a(\pm 1) = O(n^{-1/4})$, instead of $O(n^{-1/2})$ as in the last subsection, using a microscopic Cole-Hopf transform, it was shown that the diffusively scaled height fluctuations converge to $\log Z(t,x)$ \cite{BG}.  In \cite{ACQ}, different initial conditions are considered, as well as importantly `exact' statistics of the Cole-Hopf solution process.

Consider now the KPZ Burgers equation,
\begin{equation}
\label{KPZ Burgers}\partial_t u \ = \ a \Delta u + b \nabla u^2 + c \nabla \dot\W_t,\end{equation}
which formally governs the gradient $u = \nabla h$ of the KPZ equation solution.
Again, the equation is ill posed.  However, since $\eta_t(x)$ is the discrete gradient of the microscopic height function, to try to derive \eqref{KPZ Burgers}, it is natural to look at the fluctuation field which represents a microscopic form of $u$.

In \cite{GJ} and \cite{gjs}, in a class of systems starting from $\nu_\rho$, with nearest-neighbor weakly asymmetric jump probability so that $a(\pm 1)= O(n^{-1/4})$ as above,  it was shown that all limit points ${\mc Z}_t$ of the field,
$${\mc Z}^n_t\ = \ \frac{1}{n^{1/4}}\sum_x \tau_{\lfloor nvt\rfloor}(\eta_{nt}(x) - \rho\big)\delta_{x/n^{1/2}},$$
  in a moving frame with a characteristic speed $vnt$, satisfy a martingale formulation of \eqref{KPZ Burgers}.    Namely,
${\mc Z}_t(H) - {\mc Z}_0(H) - c_1\int_0^t {\mc Z}_s(\Delta H) ds - c_2\A_t(H)$ is a martingale corresponding to $c\nabla \dot\W_t$.  Here, the term is defined,
$$\A_t(H) \ = \ \lim_{\epsilon\downarrow 0} \int_0^t \int \nabla H(x) {\mc Z}_s(\tau_{x}G_{\epsilon})^2 ds,$$
where $G_\epsilon$ is a smoothing of the delta mass at $0$ and $\tau_y$ is shift by $y$.  The constants $c_1$ and $c_2$ are homogenized factors reflecting the density $\rho$ and the rates of particle interactions.  Although uniqueness of a limit process has not been shown for this type of martingale formulation, it does indicate structure corresponding to \eqref{KPZ Burgers}. 

In this context, what happens in long range systems when $s(x) = O(|x|^{-(1+\alpha)}$ and $a$ is nearest-neighbor of certain strength?  When $\alpha> 2$ and $a(\pm 1) = O(n^{1/4})$ or $\alpha=2$ and $a(\pm 1) = O((n\log n)^{-1/4})$, it is a straightforward computation, following \cite{gjs}, to see that the same sort of limit behaviors hold, with different constants, as in the nearest-neighbor setting.    

Part of our motitivation then is to ask, when $0<\alpha<2$ and $a(\pm 1)$ is of certain strength, if the limits of the fluctuation field `solve' a type of fractional KPZ Burgers equation,
\begin{equation}
\label{fractional KPZ Burgers}
\partial_t u  \ = \ a\Delta^{\alpha/2}u + b \nabla u^2 + c \nabla^{\alpha/2}\dot\W_t.
\end{equation}
We comment, although there does not seem to be a `Cole-Hopf' formula to analyze \eqref{fractional KPZ Burgers}, it would be of interest to understand the equation from the point of view of Hairer's rough paths approach \cite{H2}.  In this respect, it appears \eqref{fractional KPZ Burgers} formally can be made to make sense when $\alpha>3/2$ \cite{Jeremy}.

\subsection{Sketch of main results}
\label{survey_subsection}
After having described motivations, we now describe briefly our main results.  To introduce the main ideas and to be concrete, we concentrate in the article on zero-range processes (cf. definitions in Subsection \ref{notation-results}) with jump probability $p$ such that $s(x) = O(|x|^{-(1+\alpha)})$ and $a(\cdot)$ is nearest-neighbor with varying strengths, often depending on the scaling parameter $n$.  The zero-range process is a representative system:  In principle, all of the results in the article should hold in a more general setting as in \cite{gjs}.

Also, we will focus on phenomena when $0<\alpha<2$, as the $\alpha\geq 2$ fluctuation field behavior, already mentioned above, is more standard and straightforwardly can be shown to correspond to results in the nearest-neighbor setting \cite{gjs}.   

\medskip 

Our first result (Theorem \ref{drift_thm}) sets the stage for later limits and identifies, in a fixed frame of reference, the equilibrium fluctuations of the density field, for long range zero-range systems with the same nearest-neighbor weak asymmetries as in Subsection \ref{hyd_subsection}, namely $a(\pm 1) = O(n^{-(1-1/\alpha)})$,
 as corresponding to linearizations of the hydrodynamic limits near \eqref{zr_hyd}.  The limits are two types of fractional stochastic heat equations \eqref{gen_OU_DG} and \eqref{gen_OU_drift}, one without and one with a linear drift term, depending on whether $0<\alpha<1$ or $1\leq \alpha<2$ respectively.  Such equations were perhaps first considered in the literature with respect to limits of certain branching particle systems \cite{Dawson_Gorostiza}, \cite{Dawson_Gorostiza1}.    

Next, after absorbing linear drift terms, by observing these fluctuation fields in a moving frame with a characteristic velocity, we obtain a transition point at $\alpha = 3/2$ (Theorem \ref{secondorder_thm}).  Namely, for $3/2\leq \alpha<2$, when $a(\pm 1) = O(n^{-1+3/2\alpha})$, the equilibrium fluctuation limit points satisfy a martingale formulation of a fractional stochastic Burgers equation  \eqref{gen_OU_sec}.  While for $0<\alpha<3/2$, no matter the order of the asymmetry $a(\pm 1)$, the equilibrium fluctuation limit is the unique solution of a fractional stochastic heat equation without drift.

As mentioned in Subsection \ref{occ_subsection}, this result complements the work in \cite{bgs} with respect to `local' fluctuations of the exclusion occupation time, and shows a certain `universality' of the transition point $\alpha =3/2$ with respect to `bulk' fluctuations, in a general class of zero-range systems, across process characteristics.  In particular, the presence of the `gradient of the square' term in \eqref{gen_OU_sec}, when $\alpha\geq 3/2$, is more evidence that the `strongly' asymmetric system, when $a(\pm 1)$ is a nonzero constant, is in the standard KPZ class.  In this respect, we note for the parameter $\alpha = 3/2$, the process is not weakly asymmetric but `strongly' so.  We mention also it is open to show that the martingale formulation uniquely characterizes a limit solution of \eqref{gen_OU_sec}, although it suggests much of the structure of the equation (cf. Remark \ref{rmk_sec}).

The methods of the article, as in hydrodynamics, are to develop the stochastic differential of the fluctuation field ${\mathcal Z}^n_t$ and to close the equation by averaging nonlinear rate terms in terms of the field itself.  Such averaging, in the fluctuation field context, known as a Boltzmann-Gibbs principle, has been proved in a sharp form in \cite{gjs} for nearest-neighbor models.  Taking advantage of a long range adaptation, one can pass to the limit and obtain formally different SPDEs depending on parameters.  However, to make rigorous the convergence, unlike in the nearest-neighbor setting, as the fractional Laplacian $\Delta^{\alpha/2}$ does not take the class of Schwarz class functions to itself, several technical estimates are needed as in \cite{Dawson_Gorostiza}, \cite{Dawson_Gorostiza1} which considered related limits. 

\medskip 
In the next section, we define the zero-range model and state results.  In Section \ref{proofs}, in several subsections, the main statements are proved.

\section{Models and Results}
\label{notation-results}

After defining the zero-range model and stating assumptions, we proceed to the main results.

\subsection{Notation and Assumptions}
\label{notation}

Let $\{\eta^n_t: t\geq 0\}$ be a sequence of zero-range particle systems on state space $\Omega = \N_0^\Z$ where $\N_0=\{0,1,2,\ldots\}$.  The configuration $\eta_t = \{\eta_t(x): x\in \Z\}$ specifies the particle occupation numbers $\eta_t(x)$ at sites $x\in \Z$ at time $t\geq 0$.

Define the symmetric jump probability $s=s_\alpha: \Z \rightarrow [0,1]$ by
$$s(x) \ = \ \frac{c_\alpha}{|x|^{1+\alpha}}; \ \ \ x\neq 0$$
and $s(0)=0$ for $\alpha>0$ where $c_\alpha$ is a normalization constant.  Let also $a:\Z \rightarrow \R$ be an anti-symmetric function given by
$$a(x) \ = \ \left\{\begin{array}{rl}
1 & \ {\rm for \ }x=1\\
-1& \ {\rm for \ }x=-1\\
0& \ {\rm otherwise.}\end{array}\right.
$$
For $\gamma,\beta\geq 0$, define the jump probability $p = p_{n,\gamma,\beta,m}: \Z \rightarrow [0,1]$ by 
$$p(\cdot) \ = \ s(\cdot) + \frac{\beta}{n^\gamma} a(\cdot).$$
  Here, $n$ is taken large enough, say $n\geq n_0$, so that $0<p(\pm 1)<1$ when $\gamma>0$.  Similarly, when $\gamma=0$, $\beta>0$ is fixed small enough so that $0<p(\pm 1)<1$.  Without loss of generality, we may assume $n_0=1$.

Let also $g: \N_0\rightarrow \R_+$ be the `rate' for the process, such that $g(0)=0$ and $g(k)>0$ for $k\geq 1$.  Informally, the zero-range system is described as follows:  If there are $k$ particles at a location, $g(k)$ is the rate at which one of these particles jumps.  Then, the location to where it jumps to is governed by $p$.

With respect to $p$ and $g$, the dynamics of the process is given by generator
$$L_n f(\eta) \ = \ n\sum_{x\in \Z} \sum_{y\in \Z}p(y)g(\eta(x))\big\{f(\eta^{x,x+y}) - f(\eta)\big\}.$$ 
Here, $\eta^{v,w}$ is the configuration after a particle moves from $v$ to $w$.
$$\eta^{v,w}(z) \ = \ \left\{\begin{array}{rl}
\eta(v) -1 &  \ {\rm when \ } z=v\\
\eta(w)+1 & \ {\rm when \ } z=w\\
\eta(z) & \ {\rm otherwise.}
\end{array}\right.
$$

We will assume $g$ satisfies the following contstructibility condition.
\begin{itemize}
\item[(LIP)] There is a constant $M$ such that $\sup_k | g(k+1) - g(k)| \leq M<\infty$.
\end{itemize}

Under condition (LIP), the process $\eta^n_t$ can be constructed as a Markov process on $\Omega$ with a family of invariant measures $\{\bar\nu_\theta: 0\leq \theta <\theta_*\}$ where $\theta_*=\liminf_{k\uparrow \infty} g(k)$.  These probability measures, indexed by `fugacities', are product over lattice points in $\Z$ with common marginal given by 
$$\bar\nu_\theta(\eta(x) = k) \ = \ \frac{1}{Z_\rho} \frac{ \theta^k}{g(k)!}$$
for $k\geq 0$.  Here, $g(k)! = g(k)\cdots g(1)$ when $k\geq 1$ and $g(0)! = 1$.  

It will be convenient to index these measures by `density', that is $\rho(\theta) = \int \eta(0)d\bar\nu_\theta$.  One can see that $\rho$ is a strictly increasing function of $\theta$.  Let $\theta=\theta(\rho)$ be the inverse function and define $\nu_\rho = \bar\nu_{\theta(\rho)}$ for $0\leq \rho< \rho_*$ where $\rho_* = \lim_{\theta\uparrow  \theta_*}\rho(\theta)$.  

Moreover, with respect to a fixed $\nu_\rho$, the process can be realized as a Markov process on $L^2(\nu_\rho)$ with Markov generator $L_n$ and a core of local $L^2(\nu_\rho)$ functions.   The measure $\nu_\rho$ is also invariant with respect to the adjoint process, and is reversible when $p = s$.   See \cite{Andjel} and \cite{S_extremal} for more details about construction and invariant measures of the process.  

Here, a local function is one which depends only on a finite number of occupation variables $\{\eta(x): x\in \Z\}$. 
Also, in the following, we denote by $\P_\kappa$ and $\E_\kappa$ the measure and expectation of the process when started from initial measure $\kappa$.  Also, $E_\kappa$ and ${\rm Var}_\kappa$ will denote expectation and variance with respect to $\kappa$.  

Define, for a local $f$, the
function $\tilde{f}(z) = E_{\nu_z}[f]$, when the expectation makes sense.

The mixing properties of the system will play a role in the analysis.  Consider the localized process on the interval $\Lambda_\ell = \{x\in \Z: |x|\leq \ell\}$ with $k\geq 0$ particles and generator
$$S_{n,\ell} f(\eta) \ = \ \sum_{x,y\in \Lambda_\ell}g(\eta(x))\big\{f(\eta^{x,y}) - f(\eta)\big\}s(y-x).$$
For this Markov chain, the canonical measure $\nu_{k,\ell} = \nu_\rho(\cdot| \sum_{x\in \Lambda_\ell}\eta(x) = k\}$ is reversible and invariant.  Let $\lambda_{k,\ell}$ be the spectral gap, that is the second smallest eigenvalue of $-S_n$ (with $0$ being smallest).  Denote $W(k,\ell) = \lambda^{-1}_{k,\ell}$ and note the Poincar\'e inequality
$${\rm Var}_{\nu_{k,\ell}}(f) \ \leq W(k,\ell) D_n(f,\nu_{k,\ell})$$
where $D_n$ is the canonical Dirichlet form
$$D_n(f,\nu_{k,\ell}) \ = \ \frac{1}{2}\sum_{x,y\in \Lambda_\ell} E_{\nu_{k,\ell}}\big[g(\eta(x))\big\{f(\eta^{x,y})-f(\eta)\big\}^2\big]s(y-x).
$$

We will suppose the following condition which guarantees sufficient mixing for our purposes:
\begin{itemize}
\item[(SG)] There is a constant $C=C(\rho)$ such that 
$$E_{\nu_\rho} \left[ W\left(\sum_{x\in \Lambda_\ell}\eta(x), \ell\right)^2\right]  \ \leq \ C\ell^{2\alpha}.
$$
\end{itemize}
The condition (SG) is a condition on the rate $g$.  

There is a large class of rates for which this condition holds.  Consider the process on the complete graph with vertices in $\Lambda_\ell$ and Dirichlet form
$$D^{\rm unif}_n(f) \ := \ \frac{1}{(2\ell)}\sum_{x,y\in \Lambda_\ell}g(\eta(x))\big\{f(\eta^{x,y}) - f(\eta)\big\}.$$
Often the spectral gap $\lambda_{m}$ with respect to this mean-field process is easier to estimate.  Suppose the bound $\lambda_m \geq r(k,\ell)> 0$ holds.  Then, one can derive a bound on $W(k,\ell)$ for the long range dynamics noting that
$$[(2\ell)^\alpha/c_\alpha] D_n(f, \nu_{k,\ell}) \ \geq \ D^{\rm unif}_n(f).$$  
Then,
$${\rm Var}_{\nu_{k,\ell}}(f)  \ \leq \ r(k,\ell) D^{\rm unif}_n(f) \ \leq \ \frac{(2\ell)^\alpha r(k,\ell)}{c_\alpha} D_n(f, \nu_{k,\ell})$$
which gives the estimate $W(k,\ell) \leq r(k,\ell)[(2\ell)^\alpha/c_\alpha]$.

Suitable mean-field spectral gaps, which lead to verification of (SG), have been proved for a large class of processes.  In the following, $C$ is a constant not depending on $k$ or $\ell$.

\begin{itemize}

\item When $g(k+k_0) - g(k) \geq m_0$ for all $k\geq 1$, and $k_0\in \N$ and $m_0>0$ are fixed, $r(k,\ell)\leq C$ \cite{Caputo}.

\item When $g(k) = k^\beta$ for $0<\beta<1$, $r(k,\ell)\leq  C(1 + k/\ell)^\beta$  \cite{Nagahata}.
\item When $g(k) = 1(k\geq 1)$, $r(k,\ell) \leq C(1+ k/\ell)^2$ \cite{Morris}
\end{itemize}

\subsection{Results}

Let $\S(\R)$ be the standard Schwarz space of smooth, rapidly decreasing functions equipped with the usual metric.  Let also $\S'(\R)$ be the dual space of tempered distributions on $\R$ endowed with the strong topology.  Denote by $D([0,T], \S'(\R))$ and $C([0,T], \S'(\R))$ the function spaces of cadl\'ag and continuous maps respectively from $[0,T]$ to $\S'(\R)$.  

Let also $\widehat{C}$ be the space of infinitely differentiable functions with support contained in $(-\delta_0, T)$ for some $\delta_0>0$ fixed.  The bracket will $\langle \cdot,\cdot\rangle$ denote the duality with respect to $(\S'(\R), \S(\R))$ and other pairs of spaces when the context is clear.

For $0<\alpha<2$, let now $\Y^n_t$ be the density fluctuation field, acting on functions $H\in \S'(\R)$, given by
$$\Y^n_t(H) \ = \ \frac{1}{n^{1/2\alpha}}\sum_x H\left(\frac{x}{n^{1/\alpha}}\right)\big(\eta^n_{nt}(x) - \rho\big).$$

%

Throughout this article, the initial configuration $\eta^n_0$ will be distributed according to a fixed $\nu_\rho$.  Then, from the central limit theorem, for each fixed $t\geq 0$, $\Y^n_t$ converges in distribution to $\dot\W_0$, the spatial White noise with standard covariance $\langle \dot\W_0(G), \dot\W_0(H)\rangle = \sigma^2(\rho)\int_\R G(x)H(x)dx$ where $\sigma^2(\rho) = E_{\nu_\rho}\big[\eta(0)-\rho)^2\big]$.  

Define also the space-time White noise $\dot\W_t$ with covariance
$$\left\langle \dot\W_t(G), \dot\W_s(H)\right\rangle \ = \ \delta(t-s) \sigma^2(\rho)\int_\R G(x)H(x)dx.$$

When $\beta = 0$ and $0<\alpha<2$, that is when $p=s$ and the process is symmetric, a martingale form of 
the following `equilibrium fluctuations' result was shown in \cite{Jara_longrange}.  See also Theorem \ref{drift_thm}, which when $\beta=0$, recovers this statement. 

\begin{proposition}
\label{symmetric_prop}
When $\beta = 0$ and $0<\alpha< 2$, starting from initial measure $\nu_\rho$, the sequence $\Y^n_t$, as $n\uparrow\infty$, converges in the uniform topology on $D([0,T], \S'(\R))$ to the unique process $\Y_t$ which solves the generalized Ornstein-Uhlenbeck equation
\begin{equation}
\label{gen_OU_DG}
\partial_t \Y_t \ = \ \tilde{g}'(\rho)\Delta^{\alpha/2}\Y_t + \sqrt{\tilde{g}(\rho)}\nabla^{\alpha/2} \dot\W_t.
\end{equation}
\end{proposition}

Recall the fractional Laplacian operator $\Delta^{\alpha/2}$ acting on $H\in \S(\R)$ is given by
$$(\Delta^{\alpha/2}H)(x) \ = \ \frac{1}{2}\int_{\R} s(y) \left[ H(x+y) -2H(x) + H(x-y)\right]dy.$$
For $G,H\in \S(R)$, the covariance $E_{\nu_\rho}\big[\Y_t(G)\Y_t(H)\big] = \sigma^2(\rho)\int_\R G(x)H(x)dx$, and the noise 
$$dN_t \ := \ \nabla^{\alpha/2} \dot\W_t,$$ when integrated in time, is a spatial White noise with covariance
\begin{eqnarray*}
&&E_{\nu_\rho}\left[\int_0^t\nabla^{\alpha/2}\dot \W_s(H) ds \cdot \int_0^t\nabla^{\alpha/2}\dot\W_s(G)ds\right]\\
&&\ \  = \ \sigma^2(\rho)\int_{\R}\int_{\R} s(y)\big(H(x+y)-H(x)\big)\big(G(x+y)-G(x)\big)dydx\\
&&\ \ = \ \sigma^2(\rho)\int_\R G(x) \Delta^{\alpha/2}H(x)dx.
\end{eqnarray*}
When $G=H$, we say $\|\nabla^{\alpha/2}G\|^2_{L^2(\R\times \R)} := \sigma^2(\rho)\int_\R\int_\R s(y) \big( G(x+y) - G(x)\big)^2dydx$.

Let $\{T_t: t\geq 0\}$ be the semigroup of symmetric bounded linear operators on $L^2(\R)$ generated by $\Delta^{\alpha/2}$ (cf. \cite{Ito_Mckean}).  Symbolically, \eqref{gen_OU_DG} translates to 
\begin{equation}
\label{mild_form}
\Y_t \ = \ T_tY_0 + \int_0^t T_{t-s}dN_s.\end{equation}
    
Unfortunately, $\Delta^{\alpha/2}H$ and $T_t H$ do not in general belong to $\S(\R)$, and so the middle term on the right-side of \eqref{gen_OU_DG} and also terms in \eqref{mild_form}, in weak formulations, do not make apriori sense.  However, as shown and discussed in Proposition 3.3 and Remark 3.4(a) in \cite{Dawson_Gorostiza} (see also \cite{Dawson_Gorostiza1}), suppose all terms of \eqref{gen_OU_DG} and \eqref{mild_form} make sense when integrated with respect to functions in the nuclear space $\Phi_t(x)\in \S(\R)\otimes \widehat{C}$ topologized by
norms $\|\Phi\|_k = \max_{0\leq |\ell|\leq n}\sup_{x\in \R, t\in [-\delta_0, T]}(1+ |x|^2)^k|D^\ell\Phi_t(x)|$ where the $\ell$th order derivative $D^\ell$ acts on variables $x,t$ (cf. \cite{Treves}).  That is,
\begin{equation}
\label{terms}
\int_0^T \langle \Y_t, \Delta^{\alpha/2}\Phi_t\rangle dt, \ \Big\langle \Y_0, \int_0^T T_t \Phi_tdt\Big\rangle, \ \ {\rm and \  \ }
\int_0^T \Big\langle N_s, \int_s^T T_{t-s}\partial_t \Phi_t dt\Big\rangle ds,\end{equation}
 can be seen to define $\big(\S(\R)\otimes \widehat{C}\big)'$-valued random variables.  And, suppose the equation
\begin{equation}
\label{3.3}
\int_0^T \langle \Y_t, \partial_t \Phi_t + \Delta^{\alpha/2}\Phi_t\rangle dt \ = \ - \langle \Y_0, \Phi_0\rangle + \int_0^T \langle N_t, \partial_t\Phi_t\rangle dt\end{equation}
holds, then one concludes the evolution equation
  $$
  \int_0^T \langle \Y_t, \Phi_t\rangle dt \ = \ \Big\langle \Y_0, \int_0^T T_t\Phi_tdt\Big\rangle - \int_0^T \Big\langle N_s, \int_s^T T_{t-s}\partial_t \Phi_tdt\Big\rangle ds$$
  also holds.  We remark by the Hahn Banach theorem the terms in \eqref{terms} and associated equations above can be extended to larger domains (cf. Remark 3.4(a) in \cite{Dawson_Gorostiza} and Remark 3.3 in \cite{Dawson_Gorostiza1}).  As at most one process can satisfy the evolution equation, and the process defined by the action $\int_0^T \langle \Y_t, \Phi_t\rangle dt$, determines $\Y_t$ (Lemma 2.3 in \cite{Dawson_Gorostiza1}), these facts show \eqref{3.3} has a unique solution.  
  
  Part of the proof of our later results is to argue that the limit process $\Y_t$ satisfies all of these ingredients, when the limit equation is linear.

\medskip 
When $\beta> 0$, the strength of the weak-asymmetry $\gamma$ should be specified.  It turns out $\gamma$ should depend on $\alpha$ to obtain nontrivial limits.  

\begin{theorem}
\label{drift_thm}
Starting from initial measure $\nu_\rho$, the sequence $\Y^n_t$, as $n\uparrow\infty$, converges in the uniform topology on $D([0,T], \S'(\R))$ to the unique process $\Y_t$ which solves the following generalized Ornstein-Uhlenbeck equations:

 When $\beta \geq 0$, $\gamma = 1 - 1/\alpha$, and $1\leq \alpha < 2$, 
\begin{equation}
\label{gen_OU_drift}
\partial_t \Y_t \ = \ \tilde{g}'(\rho)\Delta^{\alpha/2}\Y_t + \beta\tilde{g}'(\rho)\nabla \Y_t + \sqrt{\tilde{g}(\rho)}\nabla^{\alpha/2} \dot\W_t.
\end{equation}

 When $\beta \geq 0$ and $0<\alpha<1$, no matter the value of $\gamma\geq 0$, $\Y_t$ satisfies the symmetric process limit equation \eqref{gen_OU_DG}.

\end{theorem}

\begin{remark}\rm
Here, the long range strength parameter $\alpha$ is a transition point, expected as it is already present with respect to the associated hydrodynamic equation \eqref{zr_hyd}.  In words, when $\alpha<1$, the fluctuation field limit for the $\beta>0$ asymmetric process, whether weak asymmetric or even plainly asymmetric, is the limit for the symmetric process.  However, when $\alpha\geq 1$, for the tuned asymmetric process, with $\gamma= \gamma(\alpha)$ chosen appropriately, the limit equation involves an extra drift.    

We note a `crossover' effect is implied straightforwardly by the generator calculation leading to Theorem \ref{drift_thm}:  When $1\leq \alpha< 2$ and $\gamma> 1-1/\alpha$, the extra drift term in \eqref{gen_OU_drift} disappears, and $\Y_t$ solves the symmetric process limit equation \eqref{gen_OU_DG}.

The equation \eqref{gen_OU_drift} can be written in terms of \eqref{gen_OU_DG} by introducing a reference frame shift:  That is, let ${\mathcal Z}_t(G) = \Y_t(G(\cdot - \beta\tilde{g}'(\rho)t)$ for $G\in \S(\R)$.  Then, ${\mathcal Z}_t$ satisfies the driftless \eqref{gen_OU_DG}.  Hence, well-posedness and uniquess of the solution of \eqref{gen_OU_drift} follows from that of the symmetric process limit equation.
\end{remark}

To probe second-order effects, we now absorb the drift in \eqref{gen_OU_drift}, by observing the fluctuation field moving with a `characteristic' velocity.  Define $\Y^{n,\rightarrow}_t$, in terms of its action on $H\in \S(\R)$, as
$$\Y^{n, \rightarrow}_t(H) \ = \ \frac{1}{n^{1/2\alpha}}\sum_x H\left(\frac{x}{n^{1/\alpha}} - \frac{1}{n^{1/\alpha}}\left\{\frac{\beta \tilde{g}'(\rho)tn}{n^\gamma}\right\}\right)\big(\eta^n_t(x) - \rho\big).$$
Again, the possible limits of $\Y^{n, \rightarrow}_t$ when $\beta>0$ depend on the strength of the weak-asymmetry $\gamma$.

It will turn out, as discussed in the introduction, when $\gamma = 1-3/2\alpha$ and $3/2\leq \alpha\leq 2$, the asymmety is significant enough to introduce a `quadratic' term in the limit.  
Formally, the limits of $\Y^{n, \rightarrow}_t$ satisfy a type of (ill posed) fractional KPZ-Burgers equation,
\begin{equation}
\label{gen_OU_sec}
\partial_t \Y_t \ = \ \tilde{g}'(\rho)\Delta^{\alpha/2}\Y_t + \beta\tilde{g}''(\rho)\nabla \Y^2_t + \sqrt{\tilde{g}(\rho)}\nabla^{\alpha/2}\dot\W_t.
\end{equation}
We note if one replaces $\Delta^{\alpha/2}$ and $\nabla^{\alpha/2}$ by $\Delta/2$ and $\nabla$ respectively then the equation reduces to KPZ-Burgers equation which governs $\nabla h_t$, where $h_t$ satisfies a KPZ equation. 

To give a sense to this equation, as in \cite{gjs}, we define the notion of an `$L^2$-energy' martingale formulation of \eqref{gen_OU_sec}.  
Let $\iota: \bb R \to [0,\infty)$ be given by $\iota(z) = (1/2)1_{[-1,1]}(z)$ and, for $\varepsilon>0$, let
$\iota_\varepsilon(z) = \varepsilon^{-1} \iota(\varepsilon^{-1}z)$.
Let also $G_\varepsilon: \bb R \to [0,\infty)$ be a smooth compactly supported approximating function in $\S(\R)$ such that $\|G_\varepsilon\|^2_{L^2(\R)} \leq 2\|\iota_\varepsilon\|^2_{L^2(\R)}=\varepsilon^{-1}$ and 
$$\lim_{\varepsilon\downarrow 0}\varepsilon^{-1/2}\|G_\varepsilon -\iota_\varepsilon\|_{L^2(\R)}  \ = \ 0.$$
Such approximating functions can be found by convoluting $\iota_\varepsilon$ with smooth kernels.  
Let also, for $x\in \R$, $\tau_x$ be the shift so that $\tau_x G_\varepsilon(z) = G_\varepsilon(x+z)$.

For an $\S'(\R)$-valued process $\{\mc Y_t; t \in [0,T]\}$ and for $0\leq s\leq t\leq T$, define
\[
\mc A_{s,t}^\varepsilon(H) \ =\  \int_s^t \int\limits_{\bb R} \nabla H(x) \Big[\mc \Y_u(\tau_{-x} G_\varepsilon)\Big]^2  dx du.
\]
The process $\Y_\cdot$ satisfies the {\em $L^2$ energy condition} if
for $H\in \S(\R)$,
\begin{equation}
\label{en.cond}
\{\A_{s,t}^\varepsilon(H)\} {\rm \ is \ Cauchy \ in \ }L^2(\nu_\rho) {\rm \ as \ }\varepsilon \downarrow 0
\end{equation}
and the limit does not depend on the specific smoothing family $\{G_\varepsilon\}$.
Define the process  
$\{\mc A_{s,t}; 0\leq s\leq t\leq T\}$ given by 
$$\mc A_{s,t}(H) \ := \ \lim_{\varepsilon\downarrow 0} \mc A^\varepsilon_{s,t}(H),$$
which is $\S'(\R)$ valued (cf. p. 364-365; Theorem 6.15 
of \cite{Walsh}).

We will say that $\{\mc Y_t; t \in [0,T]\}$ is a {\em fractional $L^2$-energy solution} of \eqref{gen_OU_sec} if the following holds.
\begin{itemize}
\item[(i)] Initially, $\mc Y_0$ is a spatial Gaussian process with covariance, for $G,H \in \S(\R)$, 
$${\rm Cov}(\mc Y_0(G),\mc Y_0(H)) \ = \ \sigma^2(\rho)\int_\R G(x)H(x)dx.$$
\item[(ii)] The integral $\int_0^T \Y_s(\Delta^{\alpha/2}H_s)ds$, for $H_\cdot \in \S(\R)\otimes \widehat C$, defines an $(\S(\R)\otimes \widehat C)'$ random variable. 
\item[(iii)] The process $\{\mc Y_t; t \in [0,T]\}$ satisfies the $L^2$-energy condition \eqref{en.cond}.
\item[(iv)] The $\S'(\R)$ valued
process $\{\M_t: t\in [0,T]\}$ where
\begin{equation*}
\label{formal_conservative}
\mc M_t(H) \ :=\  \mc Y_t(H) - \mc Y_0(H) -  \tilde{g}'(\rho)\int_0^t \mc Y_s(\Delta^{\alpha/2} H) ds - \beta \tilde{g}''(\rho) \mc A_{0,t}(H)
\end{equation*}
is a continuous martingale (a Brownian motion by Levy's theorem) with quadratic variation
\[
\< \mc M_t(H)\> \ =\  \tilde{g}(\rho)t\|\nabla^{\alpha/2}H\|^2_{L^2(\R\times\R)}.
\]
\end{itemize}

\begin{theorem}
\label{secondorder_thm}
Starting from initial measure $\nu_\rho$, when $3/2\leq \alpha< 2$, $\beta>0$, and $\gamma = 1- 3/2\alpha$, the sequence $\{\Y^{n, \rightarrow}_t: t\in [0,T]\}_{n\geq 1}$ is tight in the uniform topology on $D([0,T], \S'(\R))$, and any limit point $\Y_t$ is an fractional $L^2$-energy solution of \eqref{gen_OU_sec}.

However, for $\beta\geq 0$, when $0<\alpha<3/2$, no matter the value of $\gamma\geq 0$, $\Y^{n, \rightarrow}_t$ converges in the uniform topology on $D([0,T], \S'(\R))$ to the unique process $\Y_t$ which solves the symmetric process limit equation \eqref{gen_OU_DG}.

\end{theorem}

\begin{remark}\label{rmk_sec}\rm
The result indicates a transition point when long range strength parameter $\alpha=3/2$, consistent with `local' fluctuation results in \cite{bgs} (cf. Subsection \ref{occ_subsection}).  Namely, when $\alpha<3/2$, the characterteristic velocity translated fluctuation field limit for the $\beta>0$ asymmetric process, no matter the strength of the asymmetry, is the limit for the symmetric process.  

However, when $\alpha\geq 3/2$, under an appropriate asymmetry scale $\gamma= \gamma(\alpha)$, the limit points satisfy a martingale formulation of a `fractional' KPZ-Burgers equation, involving a `quadratic gradient' term.   We remark, although quite suggestive, it is open to show this martingale formulation would characterize a unique process satisfying it. 

Again, by the proof of Theorem \ref{secondorder_thm}, there is a `crossover' effect in that, for $3/2\leq \alpha\leq 2$ and $\gamma> 1-3/2\alpha$, the `quadratic' term drops out and the sequence $\Y^{n, \rightarrow}_t$ converges in the uniform topology on $D([0,T], \S'(\R))$ to the unique solution of \eqref{gen_OU_DG}.
\end{remark}

\section{Proofs}
\label{proofs}
The arguments for Theorems \ref{drift_thm} and \ref{secondorder_thm} adapt the `hydrodynamics' scheme of \cite{gjs}, with some new features, to the long-range context, developing the stochastic differential of $\Y^n_t$ and $\Y^{n, \rightarrow}_t$ into drift and martingale terms, before analyzing their limits.  Since the arguments of the two theorems are similar, to simplify the discussion, we only prove in detail Theorem \ref{secondorder_thm}, which is most involved.

In Subsection \ref{Assoc_mart_section}, various generator actions are computed in general.  Then, in Subsection \ref{BG_statement}, a general `Boltzmann-Gibbs' principle is stated which will help close equations.  In Subsection \ref{tightness}, tightness of the processes in Theorem \ref{secondorder_thm} is shown.  In Subsection \ref{generalized_domain}, we discuss essential notions which put the fractional stochastic heat equation in \eqref{gen_OU_DG} on a firm footing.  Finally, in Subsection \ref{identification}, we identify limit points and at the end finish the proof of Theorem \ref{secondorder_thm}.

\subsection{Stochastic differentials}
\label{Assoc_mart_section}

For $H\in \S(\R)$, $x\in \Z$, $0<\alpha<2$, and $n\geq 1$, define scaled and unscaled operators:
\begin{eqnarray*}
\Delta_{x,y}^nH &=& H\left(\frac{x+y}{n^{1/\alpha}}\right) + H\left(\frac{x-y}{n^{1/\alpha}}\right) - 2H\left(\frac{x}{n^{1/\alpha}}\right), \\ 
\nabla_{x}^nH &=& \frac{n^{1/\alpha}}{2}\left\{H\left(\frac{x+1}{n^{1/\alpha}}\right) - H\left(\frac{x- 1}{n^{1/\alpha}}\right)\right\},
\end{eqnarray*}
\begin{eqnarray*}
\e^n_x H&=& n^{1/\alpha}\left\{H\left(\frac{x+1}{n^{1/\alpha}}\right) - H\left( \frac{x}{n^{1/\alpha}}\right)\right\}\\
\d^n_{x,y}H & = & H\left(\frac{x+y}{n^{1/\alpha}}\right) - H\left(\frac{x}{n^{1/\alpha}}\right).\end{eqnarray*}
Define, for $\gamma, s\geq 0$,
\begin{eqnarray}
\label{shifted_H}
&&H_{\gamma,s}(\cdot)   \ = \  H\Big( \cdot - \frac{1}{n^{1/\alpha}}\Big\lfloor\frac{2\beta\tilde{g}'(\rho)sn}{n^\gamma}\Big\rfloor\Big) \\   
&&\ \ \ \ \ \ \ \ \ \ \ \ \ {\rm and} \ \  \HW_{\gamma, s}(\cdot)  \ = \  H\Big( \cdot - \frac{1}{n^{1/\alpha}}\Big\{\frac{2\beta\tilde{g}'(\rho)sn}{n^\gamma}\Big\}\Big), \nonumber
\end{eqnarray}
functions seen in frames along $n^{-1/\alpha}\Z$ and $\R$ respectively which will be useful.

\subsubsection{Fields in a fixed frame}
We develop
\begin{eqnarray*}
L_n \Y^n_s(H) & = &  \frac{n}{2n^{1/2\alpha}}\sum_{x\in \Z}\sum_{y\in \Z}s(y) g(\eta^n_s(x)) \Delta^n_{x,y}H 
+  \frac{2n\beta}{n^{\gamma +3/2\alpha}}\sum_{x\in \Z} g(\eta^n_s(x))\nabla^n_xH.\end{eqnarray*}
Then, we have 
\begin{eqnarray*}
\M^{n}_t(H)  &:=& \ \Y^n_t(H) - \Y^n_0(H)  - \int_0^t L_n \Y^n_s(H) ds
\end{eqnarray*}
is a martingale.  In these and following calculations, we note $(\Y^n_s)^k$ and later below $F(s,\eta^n_s;H,n)^k$ for $k\geq 1$, although not local, are $L^2(\nu_\rho)$ functions which can be approximated by local ones and are in the domain of $L_n$.

Noting 
\begin{equation}
\label{s_scaling}
s(y) \ = \ \frac{1}{n^{1+1/\alpha}}s\left(\frac{y}{n^{1/\alpha}}\right),\end{equation}
we may decompose
\begin{equation}
\label{mart_decomposition_fixed}
\M^{n}_t(H) \ = \ \Y^{n}_t(H) - \Y^{n}_0(H) - \I^{n}_t(H) -\B^{n}_t(H)\end{equation}
where
\begin{eqnarray*}
\mathcal{I}^{n}_t(H) & = &
 \frac{1}{2}\int_0^t \frac{1}{n^{1/2\alpha}}\sum_{x\in \Z} \Big[\frac{1}{n^{1/\alpha}}\sum_{y\in \Z}s(y/n^{1/\alpha})\big(g(\eta^n_s(x))-\tilde{g}(\rho)\big)\Delta^n_{x,y}H\Big] ds
\\
\mathcal{B}^{n, \rightarrow}_t(H) & = & \frac{2n\beta}{n^{\gamma+3/2\alpha}}\int_0^t \sum_{x\in \Z} \big(g(\eta^n_s(x))-\tilde{g}(\rho)\big)\nabla^n_xH ds.
\end{eqnarray*}
In the last two lines, centering constants were inserted noting $\sum_x \Delta^n_{x,y} = \sum_x \nabla^n_x = 0$.

The integrand of the quadratic variation $\langle \M^{n}_t\rangle$ equals
\begin{eqnarray*}
&& L_n \big(\Y^n_s(H)\big)^2 - 2\Y^n_s(H)L_n\Y^n_s(H) \\
&&\ \ \ \ \ \ \ \ = \ \frac{1}{n^{2/\alpha}}\sum_{x\in \Z} \sum_{y\in \Z}s(y/n^{1/\alpha})g(\eta^n_s(x))(\d^n_{x,y} H)^2 \\
&&\ \ \ \ \ \ \ \ \ \ \ \ + \frac{n\beta}{n^{\gamma + 3/\alpha}}\sum_{x\in \Z} \big(g(\eta^n_s(x)) - g(\eta^n_s(x+ 1))\big) (\e^n_x H)^2.
\end{eqnarray*}
Then, $(\M^{n}_t(H))^2 - \langle \M^{n}_t(H)\rangle$ is a martingale with
\begin{eqnarray*}
\langle \M^{n}_t(H) \rangle & = & \int_0^t  \frac{1}{n^{2/\alpha}}\sum_{x\in \Z}\sum_{y\in \Z}s(y/n^{1/\alpha})g(\eta^n_s(x)) (\d^n_{x,y} H)^2ds\\
&&\ \ \ \ \ \ \ \ \ + \int_0^t \frac{n\beta}{n^{\gamma + 3/\alpha}}\sum_{x\in \Z} \big(g(\eta^n_s(x)) - g(\eta^n_s(x+1))\big) (\e^n_x H)^2 ds.\end{eqnarray*}

Since we start from the product measure $\nu_\rho$, by stationarity and Burkholder-Davis-Gundy inequality ($E[\M^4(t)] \leq E[\langle \M(t)\rangle^2]$), we have, given $0<\alpha<2$, the second term in the quadratic variation above is negligible.  We have
\begin{eqnarray*}
\E_{\nu_\rho}\big[\big(\M^{n}_t(H)-\M^{n}_s(H)\big)^4\big]  \ \leq \ C(\beta, g,H)
|t-s|^2. 
\end{eqnarray*}

\subsubsection{Fields in a moving frame}

Let $F(s, \eta^n_s;H,n) = \Y^{n, \rightarrow}_s(H)$
and write, as before in the fixed frame,
\begin{eqnarray*}
L_nF(s, \eta^n_s; H,n) & = & \frac{n}{2n^{1/2\alpha}}\sum_{x\in \Z}\sum_{y\in \Z}s(y) g(\eta^n_s(x)) \Delta^n_{x,y}\HW_{\gamma, s} \\
&&\ \ \ \ \ \ \ + \ \frac{2n\beta}{n^{\gamma +3/2\alpha}}\sum_{x\in \Z} g(\eta^n_s(x))\nabla^n_x\HW_{\gamma, s}.\end{eqnarray*}
Also,
\begin{eqnarray*}
\frac{\partial}{\partial_s}F(s,\eta^n_s; H,n) & = & \Big\{\frac{-2\beta\tilde{g}'(\rho)n}{n^\gamma}\Big\} \frac{1}{n^{3/2\alpha}}\sum_{x\in \Z} \nabla \HW_{\gamma, s}\left(\frac{x}{n}\right)\left(\eta^n_s(x)-\rho\right).
\end{eqnarray*}
  Then,
\begin{eqnarray*}
\M^{n, \rightarrow}_t(H)  &:=& \ F(t, \eta^n_t; H,n) - F(0, \eta^n_0; H,n) \\
&&\ \ \ \ \ \ \ \ \ \ - \int_0^t \frac{\partial}{\partial_s}F(s,\eta^n_s; H,n) + L_n F(s, \eta^n_s; H,n) ds
\end{eqnarray*}
is a martingale.

Making use of \eqref{s_scaling}, we write
\begin{equation}
\label{mart_decomposition}
\M^{n, \rightarrow}_t(H) \ = \ \Y^{n, \rightarrow}_t(H) - \Y^{n, \rightarrow}_0(H) - \I^{n, \rightarrow}_t(H) -\B^{n, \rightarrow}_t(H) -\K^{n, \rightarrow}_t(H)\end{equation}
where
\begin{eqnarray*}
\mathcal{I}^{n, \rightarrow}_t(H) & = &
 \frac{1}{2}\int_0^t \frac{1}{n^{1/2\alpha}}\sum_{x\in \Z} \Big[\frac{1}{n^{1/\alpha}}\sum_{y\in \Z}s(y/n^{1/\alpha})\big(g(\eta^n_s(x))-\tilde{g}(\rho)\big)\Delta^n_{x,y}H_{\gamma,s}\Big] ds
\\
\mathcal{B}^{n, \rightarrow}_t(H) & = & \frac{2n\beta}{n^{\gamma+3/2\alpha}}\int_0^t \sum_{x\in \Z} \big(g(\eta^n_s(x))-\tilde{g}(\rho) - \tilde{g}'(\rho)(\eta^n_s(x)-\rho)\big)\nabla^n_xH_{\gamma, s} ds\\
\K^{n, \rightarrow}_t(H) &=& \int_0^t\Big[\frac{1}{n^{1/2\alpha}}\sum_{x\in \Z}\Big[\frac{1}{n^{1/\alpha}}\sum_{y\in \Z}s(y/n^{1/\alpha})\kappa^{n,1}_{x,y}(H,s)\big(g(\eta^n_s(x))-\tilde{g}(\rho)\big)\Big]\\ 
&& + \frac{2n\beta}{n^{\gamma +3/2\alpha}}\sum_{x\in \Z}\kappa^{n,2}_x(H,s)\big(g(\eta^n_s(x))-\tilde{g}(\rho) - \tilde{g}'(\rho)(\eta^n_s(x)-\rho)\big)\Big]ds.
\end{eqnarray*}
Here, as $\sum_x\Delta^n_{x,y}H_{\gamma,s} = \sum_x\nabla^n_xH_{\gamma,s}=0$, centering constants were introduced in $\I^{n, \rightarrow}_t$ and $\B^{n, \rightarrow}_t$.  By Taylor expansion,
\begin{eqnarray*}
&&\kappa^{n,1}_{x,y}(H,s) \ = \ \Delta^n_{x,y} \big(\HW_{\gamma,s} - H_{\gamma,s}\big)\nonumber\\
&&\ \   = \ O(n^{-1/\alpha})\cdot \Delta^n_{x,y} H'_{\gamma,s}\nonumber\\
&&\ \ \  + O(n^{-2/\alpha})\cdot \Big[H^{(4)}_{\gamma,s}((x+y+z_1)/n^{1/\alpha})\nonumber\\
&&\ \ \ \ \ \ \ \ \  + H^{(4)}_{\gamma,s}((x-y+z_2)/n^{1/\alpha}) + 2H^{(4)}_{\gamma,s}((x+z_3)/n^{1/\alpha})\Big] \nonumber
\end{eqnarray*}
and
\begin{eqnarray*}
\kappa^{n,2}_x(H,s) &=& 
O( n^{-1/\alpha})\cdot \Delta H_{\gamma,s}(x/n^{1/\alpha}) + O(n^{-2/\alpha})\cdot H'''_{\gamma,s}(z_4/n)
\label{K_bound}
\end{eqnarray*}
where $|z_k|\leq 1$ for $1\leq k\leq 4$.

As in the fixed frame calculation, 
$(\M^{n, \rightarrow}_t(H))^2 - \langle \M^{n, \rightarrow}_t(H)\rangle$ is a martingale with
\begin{eqnarray*}
\langle \M^{n, \rightarrow}_t(H) \rangle & = & \int_0^t  \frac{1}{n^{2/\alpha}}\sum_{x\in \Z}\sum_{y\in \Z}s(y/n^{1/\alpha})g(\eta^n_s(x)) (\d^n_{x,y} \HW_{\gamma,s})^2ds\\
&&\ \ \ \ \  + \int_0^t \frac{n\beta}{n^{\gamma + 3/\alpha}}\sum_{x\in \Z} \big[g(\eta^n_s(x)) - g(\eta^n_s(x+1))\big] (\e^n_x \HW_{\gamma,s})^2 ds.\end{eqnarray*}

Also, we have the bound, as in the fixed frame, 
\begin{eqnarray}
\label{quad_var_bound}
\E_{\nu_\rho}\big[\big(\M^{n, \rightarrow}_t(H)-\M^{n, \rightarrow}_s(H)\big)^4\big]  \ \leq \ C(\beta, g,H)
|t-s|^2. 
\end{eqnarray}

\subsection{Boltzmann-Gibbs principle} \label{BG_statement}

We will need to approximate terms in the stochastic differential of $\Y^{n, \rightarrow}_t$ in order to close and recover limiting equations.  The main tool for this approximation is the `Boltzmann-Gibbs principle'.
Define
$$\big(\eta^n_s)^{(\ell)}(x) \ := \ \frac{1}{2\ell+1}\sum_{y\in \Lambda_\ell}\eta^n_s(x+y)$$
and $\sigma^2_\ell(\rho) = E_{\nu_\rho} [(\eta^{(\ell)}(0) - \rho)^2]$ for $\ell\geq 1$.

\begin{proposition}
 \label{gbg_L2}

Suppose $0<\alpha<2$.  Let $f$ be a local $L^5(\nu_\rho)$ function supported on sites $\Lambda_{\ell_0}$ such that $\tilde{f}(\rho) =\tilde{f}'(\rho)=0$.
There exists a constant $C=C(\rho, \ell_0)$
such that, for $t\geq 0$, $\ell\geq \ell^3_0$ and $h\in \ell^1(\Z)\cap \ell^2(\mathbb{Z})$,
\begin{align*}
& \bb E_{\nu_\rho}\Big[ \sup_{0\leq t\leq K}\Big( \int_0^t  \sum_{x\in{\mathbb{Z}}} \Big(\tau_x f(\eta^n_s)
   - \frac{\tilde{f}''(\rho)}{2} \Big\{\Big(\big(\eta^n_s)^{(\ell)}(x)-\rho\Big)^2-\frac{\sigma^2_\ell(\rho)}{2\ell+1} \Big\}\Big) h(x)ds\Big)^2 \Big]\\
&\ \ \ \ \ \ \ 
    \leq \ C\|f\|^2_{L^5(\nu_\rho)}\bigg(\frac{K \ell^{\alpha -1}}{n^{1-1/\alpha}}\Big(\frac{1}{n^{1/\alpha}}\sum_{x\in \Z}h^2(x)\Big) + \frac{K^2n^{2/\alpha}}{\ell^{3}}\Big(\frac{1}{n^{1/\alpha}}\sum_{x\in{\mathbb{Z}}}|h(x)|\Big)^2\bigg).
\end{align*}

On the other hand, when only $\tilde{f}(\rho)=0$ is known,
\begin{align*}
& \bb E_{\nu_\rho}\Big[ \sup_{0\leq t\leq K}\Big( \int_0^t  \sum_{x\in{\mathbb{Z}}} \Big(\tau_x f(\eta^n_s) - \tilde{f}'(\rho)\Big\{\big(\eta^n_s\big)^{(\ell)}(x) - \rho\Big\} h(x)ds\Big)^2\Big]\\
&\ \ \ \ \ \ \ \ 
    \leq \ C\|f\|^2_{L^5(\nu_\rho)}\bigg(\frac{K \ell^\alpha}{n^{1-1/\alpha}}\Big(\frac{1}{n^{1/\alpha}}\sum_{x\in \Z}h^2(x)\Big) + \frac{K^2n^{2/\alpha}}{\ell^{2}}\Big(\frac{1}{n^{1/\alpha}}\sum_{x\in{\mathbb{Z}}}|h(x)|\Big)^2\bigg).
\end{align*}
\end{proposition}

We remark Proposition \ref{gbg_L2} is an improvement of Theorem 3.2 in \cite{gjs}, which applies to zero-range processes and did not have the supremum `$\sup_{0\leq t\leq \tau}$' inside the expectation.  However the proof of Proposition \ref{gbg_L2} follows straightforwardly from the proof of Theorem 3.2 in \cite{gjs}, noting the following two comments:  

(1) The first step of the proof of Theorem 3.2 in \cite{gjs} is to use the $H_{-1}$-norm inequality, stated as Proposition 4.2 in \cite{gjs},
$$\E_{\nu_\rho}\Big[ \Big(\int_0^t f(\eta^n(s))ds\Big)^2\Big] \ \leq \ 20t\|f\|^2_{-1,n},$$
where $f$ is a local function with mean-zero $\E_{\nu_\rho}[f]=0$, and $\|f\|_{-1,n}$ is its $H_{-1}$ norm.  However, the same inequality is true if one introduces a supremum inside the expectation; that is,
$$\E_{\nu_\rho}\Big[\sup_{0\leq t\leq K} \Big(\int_0^t f(\eta^n(s))ds\Big)^2\Big] \ \leq \ 20K \|f\|^2_{-1,n};$$
see Lemma 4.3 in \cite{CLO} or Theorem 2.2 in \cite{SVY}.

(2) The proof in \cite{gjs}, in diffusive scale, now goes on to compute various $H_{-1}$-norms leading to the right-hand side.  Although in the present context, the long-range dynamics introduces an $\alpha$-dependent space-time and localized spectral gap scalings, straightforwardly applying the argument in \cite{gjs} in these scales, Proposition \ref{gbg_L2} is recovered.

\subsection{Tightness}
\label{tightness}
We now prove tightness of the fluctuation fields in Theorem \ref{secondorder_thm}, using this Boltzmann-Gibbs principle.

\begin{proposition}
\label{stationary_tightness}
Starting from $\nu_\rho$, with respect to the range of parameters in Theorem \ref{secondorder_thm}, the sequences $\{\Y^{n, \rightarrow}_t: t\in [0,T]\}_{n\geq 1}$, $\{\M^{n, \rightarrow}_t: t\in [0,T]\}_{n\geq 1}$, $\{\I^{n, \rightarrow}_t: t\in [0,T]\}_{n\geq 1}$, $\{\B^{n, \rightarrow}_t: t\in [0,T]\}_{n\geq 1}$, $\{\K^{n, \rightarrow}_t: t\in [0,T]\}$ and $\{\langle \M^{n, \rightarrow}_t\rangle: t\in [0,T]\}_{n\geq 1}$ are tight in the uniform topology on $D([0,T], \S'(\R))$.
\end{proposition}

\begin{proof}
 By Mitoma's criterion \cite{Mitoma}, for each $H\in \S(\R)$,
it is enough to show tightness of $\{\Y^{n, \rightarrow}_t(H); t\in [0,T]\}_{n\geq 1}$, $\{\M^{n, \rightarrow}_t(H): t\in [0,T]\}_{n\geq 1}$, $\{\I^{n, \rightarrow}_t(H): t\in [0,T]\}_{n\geq 1}$, $\{\B^{n, \rightarrow}_t(H): t\in [0,T]\}_{n\geq 1}$, $\{\K^{n, \rightarrow}_t(H): t\in [0,T]\}$ and $\{\langle \M^{n, \rightarrow}_t(H)\rangle: t\in [0,T]\}_{n\geq 1}$ in the uniform topology.  Note that all initial values vanish, except $\Y^{n, \rightarrow}_0(H)$. 

Tightness of $$\Y^{n, \rightarrow}_t(H)\ = \ \Y^{n, \rightarrow}_0(H) + \I^{n, \rightarrow}_t(H)  + \B^{n, \rightarrow}_t(H) + \K^{n, \rightarrow}_t(H)+\M^{n, \rightarrow}_t(H),$$
 is accomplished by showing each term is tight.  All initial values of the constituents vanish except $\Y^{n, \rightarrow}_0(H)$, which is tight as it converges weakly to a Gaussian random variable given that we start from $\nu_\rho$.

Tightness of the martingale term follows from Doob's inequality:
\begin{eqnarray*}
&&\P_{\nu_\rho}\Big( \sup_{\stackrel{|t-s|\leq \delta}{0\leq s,t\leq T}} |\M^{n, \rightarrow}_t(H) - \M^{n, \rightarrow}_s(H)|>\varepsilon\Big)\\
&& \ \ \ \ \ \ \leq \ 
 C \varepsilon^{-4}\delta^{-1} \E_{\nu_\rho}\big[ \big(\M_\delta^{n, \rightarrow}(H)\big)^4\big] \ \leq \ C(\beta, \varepsilon, g, H)\delta.
\end{eqnarray*}

The proof of tightness for $\B^{n, \rightarrow}_t(H)$ makes use of the 
the Boltzmann-Gibbs Proposition \ref{gbg_L2}.
Since $V(\eta(x)) = g(\eta(x)) -\tilde{g}(\rho) - \tilde{g}'(\rho)(\eta(x) - \rho)$ is a function of single site, $\ell_0=1$ in the application of Proposition \ref{gbg_L2}.  One obtains for $\ell> 1$ that
\begin{equation}
\label{BG_tightness}\E_{\nu_\rho}\Big[\sup_{0\leq s\leq t}(\B_s^{n, \rightarrow}(H))^2\Big] \ \leq \ \frac{C(\rho, \beta, g, H)}{n^{2\gamma +3/\alpha -2}}\Big\{\frac{t\ell^{\alpha -1}}{n^{1-1/\alpha}} + \frac{t^2n^{2/\alpha}}{\ell^3} + \frac{t^2 n^{2/\alpha}}{\ell^2}\Big\}.\end{equation}

Indeed, note $\tilde{V}(\rho) = \tilde{V}'(\rho) = 0$ and $\tilde{V}''(\rho) = \tilde{g}''(\rho)$.  Writing $\B^{n, \rightarrow}_t(H) = 2\beta n^{1-\gamma-3/2\alpha}\int_0^t \sum_x \big(\nabla^n_x H_{\gamma,s}\big) \tau_x V(\eta_s)ds$, by translation-invariance of $\nu_\rho$, we may replace $\nabla^n_x H_{\gamma,s}$ by $\nabla^n_xH$ to estimate
\begin{eqnarray*}
&&\E_{\nu_\rho} \Big[\Big(\B^{n, \rightarrow}_t(H)\\
&&\ \ \ \ \  - \frac{2\beta \tilde{g}''(\rho) n}{n^{\gamma +3/2\alpha}}\int_0^t \sum_{x\in \Z} (\nabla^n_x H_{\gamma,s})\Big\{\Big(\big(\eta_s^n\big)^{(\ell)}(x) - \rho\Big)^2 - \frac{\sigma^2_\ell(\rho)}{2\ell +1}\Big\} ds\Big)^2\Big] \nonumber\\
&&  \leq \ \frac{C(\rho,\beta,g)}{n^{2\gamma + 3/\alpha -2}} \Big\{ \frac{t\ell^{\alpha -1}}{n^{1-1/\alpha}} + \frac{t^2n^{2/\alpha}}{\ell^{3}}\Big\}\\
&&\ \ \ \ \ \ \times \Big[\Big(\frac{1}{n^{1/\alpha}}\sum_{x\in \Z} (\nabla^n_xH)^2\Big) + \Big(\frac{1}{n^{1/\alpha}}\sum_{x\in\Z}|\nabla^n_xH|\Big)^2\Big].\end{eqnarray*}
However, noting the fourth moment, $E_{\nu_\rho}[(\eta^{(\ell)}(0) - \rho)^4] \leq C\ell^{-2}$, which after squaring the integral gives the third term on the right-side of \eqref{BG_tightness}.

\medskip
For the sequence in Theorem \ref{secondorder_thm}, when $3/2\leq\alpha<2$ and $\gamma = 1-3/2\alpha$, or  $1<\alpha<3/2$, choose $\ell = t^{1/(\alpha + 1)}n^{1/\alpha}>1$.  One has $\E_{\nu_\rho}\big[(\B_t^{n, \rightarrow}(H))^2\big] \leq Ct^{2\alpha/(\alpha+1)}$  where the exponent $2\alpha/(\alpha +1)>1$.  
However, when $\ell = t^{1/(\alpha+1)}n^{1/\alpha}\leq 1$, by squaring, taking expectation and using independence under $\nu_\rho$, one gets a similar bound:
$$\E_{\nu_\rho}\big[(\B^{n, \rightarrow}_t(H))^2\big] \leq C(\rho, \beta, g) t^2n^{1/\alpha}\big[\frac{1}{n^{1/\alpha}}\sum_x |\nabla^n_x H|^2\big] \leq C(\rho, \beta, g, H)t^{\frac{2\alpha +1}{\alpha + 1}}.$$

One may now apply the Kolmogorov-Centsov criterion and stationarity of $\nu_\rho$ to obtain tightness of $\B^{n,\rightarrow}_t$ in these cases.

However, when $0<\alpha\leq 1$, since the exponent $2\alpha/(\alpha+1)\leq 1$, we use the following argument.
 By stationarity of $\nu_\rho$, the standard technique of dividing into subintervals of size $\delta^{-1}$, $\B^{n,\rightarrow}_0(H)=0$, and \eqref{BG_tightness}, for $\ell>1$, we obtain
\begin{eqnarray}
\label{tightness_help}
&&\P_{\nu_\rho}\left( \sup_{\substack{|s-t|\leq \delta\\0\leq s,t\leq T}} \big|\B^{n,\rightarrow}_t(H) - \B^{n,\rightarrow}_s(H)\big|>
\varepsilon\right) \ \leq \ \frac{3T}{\delta}\P_{\nu_\rho}\left( \sup_{0\leq t\leq \delta} \big|\B^{n,\rightarrow}_t(H)\big|>
\varepsilon\right) \nonumber\\
&& \ \ \ \leq \ \frac{3T\delta^{-1}}{\varepsilon^2} \E_{\nu_\rho}\left[\sup_{0\leq t\leq \delta} \big|\B^{n,\rightarrow}_t(H)\big|^2\right]
\ \leq \  \frac{C}{n^{2\gamma +\frac{3}{\alpha} -2}}\Big\{\frac{\ell^{\alpha -1}}{n^{1-\frac{1}{\alpha}}} + \frac{\delta n^{2/\alpha}}{\ell^2}\Big\}.
\end{eqnarray}
Choosing now $\ell = n^{1/\alpha}$, we see \eqref{tightness_help} vanishes as $n\uparrow\infty$. 

The tightness arguments for $\I^{n, \rightarrow}_t(H)$, $\langle \M^{n, \rightarrow}_t(H)\rangle$ and $\K^{n, \rightarrow}_t(H)$ are simpler and follow by squaring all terms, using independence and stationarity under $\nu_\rho$, and Kolmogorov-Centsov criterion.
For instance,
\begin{eqnarray*}
\E_{\nu_\rho} \big[(\I^{n, \rightarrow}_t)^2\big] & \leq & C(\rho,g)t^2\Big\{\frac{1}{n^{1/\alpha}}\sum_x \Big[\frac{1}{n^{1/\alpha}}\sum_y s(y/n^{1/\alpha}) \Delta^n_{x,y}H\Big]^2\Big\} \\
& \leq &C(\rho, g, H, \alpha)t^2.\end{eqnarray*}
\end{proof}

\subsection{Generalized domains}
\label{generalized_domain}
The goal of this section is to state that terms in \eqref{terms} have definition and are well-defined.  Recall the discussion after Proposition \ref{symmetric_prop}.  We now recall some notions from \cite{Dawson_Gorostiza1} and \cite{Dawson_Gorostiza}.
For $p>0$, define the Banach space 
$$C_{p,0}(\R) \ = \ \Big\{\phi\in C(\R): \phi/\phi_p\in C_0(\R)\Big\},$$
$\phi_p(x) = (1+|x|^2)^{-p}$ and $C_0(\R)$ is the set of functions vanishing at infinity.  The norm on $C_{p,0}$ is $\|\phi\|_p = \sup_x |\phi(x)/\phi_p(x)|$.

We now state part of Proposition 2.1 in \cite{Dawson_Gorostiza} and Lemma 2.5 in \cite{Dawson_Gorostiza1}:  For $1/2<p<(1+\alpha)/2$ and $t\geq 0$,
\begin{itemize} 
\item The space $C_{p,0}(\R)$ and its dual $C'_{p,0}(\R)$ are intermediate in that $\S(\R)\subset C_{p,0}(\R)\subset L^2(\R)\subset C'_{p,0}(\R)\subset \S'(\R)$.
\item $\S(\R)$ is densely and continuously embedded in $C_{p,0}(\R)$.
\item $\Delta^{\alpha/2}, T_t: \S(\R) \rightarrow C_{p,0}(\R)$ are continuous linear mapings.
\item $t\mapsto T_t\phi$ is a continuous map in $C_{p,0}(\R)$ for each $\phi\in \S(\R)$.
 \end{itemize}
 
Let now $\{\Y_t: t\in [0,T]\}$ be a limit point of either $\{\Y^{n}_t: t\in [0,T]\}$ or $\{\Y^{n, \rightarrow}_t: t\in [0,T]\}$ in the uniform topology on $D([0,T], \S'(\R))$.
Hence, $\Y_\cdot$ is continuous and we can extend, for $\Phi_t\in \S(\R)\otimes \widehat{C}$, the object $\Y_t(\Phi_t)= \langle \Y_t, \Phi_t\rangle$ for each fixed $t\in [0,T]$ and also $\int_0^T \langle \Y_t, \Phi_t\rangle dt$.    
By density of $\S(\R)$ in $C_{p,0}(\R)$, $\S(\R)\otimes \widehat{C}$ is dense in $C_{p,0}(\R)\otimes \widehat{C}$ (cf. \cite{Treves}).

\begin{proposition} 
\label{nuclear}
The $L^2(\nu_\rho)$ limits
$$\lim_{k\uparrow\infty} \int_0^T \langle \Y_t, \Phi^k_t\rangle dt, \ \ \lim_{k\uparrow\infty}\int_0^T \langle N_t, \Phi_t\rangle dt, \ \ {\rm and \ \ } \lim_{k\uparrow\infty} \langle \Y_0, \Psi^k\rangle,$$
 where $\{\Phi^k_t\}\subset \S(\R)\otimes \widehat{C}$ approximates $\Phi_t\in C_{p,0}(\R)\otimes \widehat{C}$ and $\Psi^k\in \S(\R)$ approximates $\Psi\in C_{p,0}(\R)$, are well-defined and do not depend on the approximating sequences.  

As a consequence, all terms in \eqref{terms} are linear continuous random functionals on $\S(\R)\otimes \widehat{C}$ and therefore define unique $(\S(\R)\otimes \widehat{C})'$-valued random variables.
\end{proposition}

\begin{proof}
The argument is the same as the first part of the proof of Theorem 4.1 on pages 59-61 in \cite{Dawson_Gorostiza} as applied to the Gaussian process $\Y_t$ which has covariance $\langle \Y_t(G), \Y_t(H)\rangle = \sigma^2(\rho)\int_\R G(x)H(x)dx$ and Gaussian noise field $\N_t$ with covariance
$\langle N_s(G), N_t(H)\rangle = \sigma^2(\rho)\min\{s,t\}\int_\R G(x) \Delta^{\alpha/2}H(x)dx$ for $G,H\in \S(\R)$.  To give the main idea, we give the proof that $\int_0^T\langle \Y_t, \Delta^{\alpha/2}\Phi_t\rangle dt$ for $\Phi_t\in \S(\R)\otimes \widehat{C}$ defines an $(\S(\R)\otimes \widehat{C})'$-valued random variable.  See \cite{Dawson_Gorostiza} for more details and arguments for the other terms.

We first show $\lim_{k\uparrow\infty}\int_0^T \langle \Y_t, \Phi^k_t\rangle dt$ is well defined.  
Write
\begin{eqnarray}
\label{approximation}
&&\E_{\nu_\rho} \Big| \int_0^T \langle \Y_t, \Phi^k_t\rangle dt - \int_0^T \langle \Y_t, \Phi^\ell_t \rangle dt \Big|^2  
\ = \ \E_{\nu_\rho} \Big| \int_0^T \langle \Y_t, \Phi^k_t - \Phi^\ell_t\rangle dt \Big|^2 \\
&&\ \ \ \ \ \ \ \ \ \leq \ T^2\sigma^2(\rho)\sup_{0\leq t\leq T}\int_\R \big(\Phi^k_t - \Phi^\ell_t)^2(x) dx\nonumber\\
&&\ \ \ \ \ \ \ \ \ \leq \ T^2 \sigma^2(\rho)\sup_{0\leq t\leq T}\|\Phi^k_t - \Phi^\ell_t\|^2_p \int_\R (1+ |x|^2)^{2p}dx.\nonumber
\end{eqnarray}
Hence, $\big\{\int_0^T \langle \Y_t, \Phi^k_t\rangle dt\big\}$ is a $L^2$-Cauchy sequence.  The limit does not depend on the approximation taken, and is linear and continuous in $\Phi_t\in C_{p,0}(\R)\otimes \widehat{C}$.  

Therefore, as $\Delta^{\alpha/2}:\S(\R)\rightarrow C_{p,0}(\R)$ is continuous, the maps $\Phi_t\in \S(\R)\otimes \widehat{C} \mapsto \Delta^{\alpha/2}\Phi_t\in C_{p,0}(\R)\otimes \widehat{C}$ and $\Phi_t\in \S(\R)\otimes \widehat{C} \mapsto \int_0^T \langle \Y_t, \Delta^{\alpha/2}\Phi_t\rangle dt$ are linear and continuous, the last being a linear continuous random functional.  Since $\S(\R)\otimes \widehat{C}$ is a nuclear space, by Ito's regularization theorem (cf. Lemma 2.4 in \cite{Dawson_Gorostiza1}), there is a unique $(\S(\R)\otimes \widehat{C})'$-valued random variable corresponding to the functional.
\end{proof}

\subsection{Identification}
\label{identification}
We now identify the structure of the limit points with respect to Theorem \ref{secondorder_thm}.
 Let $\mu^n$ be the distribution of
$$\left( \Y_t^{n,\rightarrow}, \M^{n,\rightarrow}_t, \I^{n, \rightarrow}_t, \B^{n, \rightarrow}_t, \K^{n, \rightarrow}_t,  \langle \M^{n, \rightarrow}_t\rangle : t\in[0,T]\right).$$
Suppose $n'$ is a subsequence where $\mu^{n'}$ converges to a limit point $\mu$.  Let also $\Y_t$, $\M_t$, $\I_t$, $\B_t$, $\K_t$ and $\D_t$ be the respective limits in distribution of the components.
Since tightness (Proposition \ref{stationary_tightness}) is shown in the uniform topology on $D([0,T], \S'(\R))$, we have that $\Y_t$, $\M_t$, $\I_t$, $\B_t$, $\K_t$ and $\D_t$ have a.s. continuous paths.

Let $G_\varepsilon:\bb R \rightarrow [0,\infty)$ be a smooth compactly supported function for $0<\varepsilon\leq 1$ which approximates $\iota_\varepsilon(z) = \varepsilon^{-1}1_{[-1,1]}(z\varepsilon^{-1})$ as mentioned before Theorem \ref{secondorder_thm}: $\|G_\varepsilon\|^2_{L^2(\R)}\leq 2\|\iota_\varepsilon\|^2_{L^2(\R)}=\varepsilon^{-1}$ and $\lim_{\varepsilon\downarrow 0}\varepsilon^{-1/2}\|G_\varepsilon - \iota_\varepsilon\|_{L^2(\R)}=0$.
For $\alpha>0$, define
$$\A^{n,\varepsilon,\rightarrow}_{s,t}(H) \ := \ \int_s^t \frac{1}{n^{1/\alpha}}\sum_{x\in \Z} (\nabla^n_x H)\big[\tau_x \Y^{n, \rightarrow}_u(G_\varepsilon)\big]^2du.$$
For fixed $0<\varepsilon\leq 1$, the transformation $\pi_\cdot \mapsto \int_s^t du \int dx\big(\nabla H(x)\big)\big\{\pi_u(\tau_{-x} G_\varepsilon)\big\}^2$ is continuous in the uniform topology on $D([0,T];\S'(\R))$.  Then, in distribution, 
\begin{eqnarray*}
\lim_{n'\uparrow\infty}\A^{n',\varepsilon, \rightarrow}_{s,t}(H) & =& \int_s^t du\int dx\big(\nabla H(x)\big)\big\{\Y_u(\tau_{-x} G_\varepsilon)\big\}^2  \ =: \ \A^{\varepsilon}_{s,t}(H).\end{eqnarray*}

\begin{proposition}
\label{stat_lemma1} Consider the systems in Theorem \ref{secondorder_thm}.  Recall the initial distribution is $\nu_\rho$ and $t\in [0,T]$.

(1) When $3/2\leq \alpha<2$ and $\gamma = 1 -3/2\alpha$, there is a constant $C=C(\beta,\rho, g)$ such that
\begin{eqnarray*}
&&\lim_{n\uparrow\infty}\E_{\nu_\rho} \Big[ \Big|\B^{n, \rightarrow}_t(H) - \beta\tilde{g}''(\rho)\A^{n,\varepsilon, \rightarrow}_{0,t}(H)\Big|^2\Big] \\
&&\ \ \ \ \ \  \leq \ Ct \Big(\varepsilon^{\alpha-1} + \varepsilon^{-1}\|G_\varepsilon - \iota_\varepsilon\|^2_{L^2(\R)}\Big)  \Big[\|\nabla H\|^2_{L^2(\R)} + \|\nabla H\|^2_{L^1(\R)}\Big].
\end{eqnarray*}
Then, in $L^2(\P_{\nu_\rho})$,
$A^\varepsilon_{0,t}(H)$ is a Cauchy $\varepsilon$-sequence.  Hence,
$$\beta \tilde{g}''(\rho) \A_{0,t}(H) \ := \ \lim_{\varepsilon\downarrow 0} \beta \tilde{g}''(\rho)\A^\varepsilon_{0,t}(H)
 \ = \ \B_t(H).$$
Also, $\A_{s,t}(H) \stackrel{d}{=} \A_{0,t-s}(H)$ does not depend on the specific family $\{G_\varepsilon\}$.

(2) On the other hand, when $0<\alpha<3/2$, we have $\lim_{n\uparrow\infty}\B^{n,\rightarrow}_t(H) = \B_t(H) = 0$ in $L^2(\P_{\nu_\rho})$.

(3) When $0<\alpha<2$,
\begin{eqnarray*}
&&\lim_{n\uparrow\infty} \E_{\nu_\rho}\Big[\Big |\I^{n, \rightarrow}_t(H) - \tilde{g}'(\rho)\int_0^t \Y^{n, \rightarrow}_s(\Delta^{\alpha/2} H)ds\Big|^2\Big] \ = \ 0\\
&&\lim_{n\uparrow\infty} \E_{\nu_\rho}\Big[\Big| \langle \M^{n, \rightarrow}_t(H)\rangle - \tilde{g}(\rho)t\|\nabla^{\alpha/2} H\|^2_{L^2(\R\times\R)}\Big|^2\Big] \ = \ 0\\
&&\lim_{n\uparrow\infty} \E_{\nu_\rho}\Big[\Big| \K^{n, \rightarrow}_t(H) \Big |^2\Big] \ = \ 0.
\end{eqnarray*}
Then, in $L^2(\P_{\nu_\rho})$, $\K_t(H)= 0$,  $D_t(H)  =  \tilde{g}(\rho)t\|\nabla^{\alpha/2} H\|^2_{L^2(\R\times \R)}$,
and $$\I_t(H) \ = \ \tilde{g}'(\rho)\int_0^t\Y_s(\Delta^{\alpha/2} H)ds.$$
Moreover,
$\M_t(H)$ is a continuous martingale with quadratic variation $D_t(H)$, and hence by Levy's theorem $\M_t$ is a version of the noise in \eqref{gen_OU_DG}. 
\end{proposition}

\begin{proof}   We first verify (1) and (3).  Suppose the limit display for $\B^{n,\rightarrow}_t(H)$ holds.  By a Fatou's lemma, we conclude
$\E_{\nu_\rho}\big[\big|\B_t(H) - \beta\tilde{g}''(\rho))\A^\varepsilon_{0,t}(H)\big|^2\big]\ \leq \ Ct\big[\varepsilon^{\alpha -1} + \varepsilon^{-1}\|G_\varepsilon - \iota_\varepsilon\|^2_{L^2(\R)}\big]$.
  Therefore, $\A^\varepsilon_{0,t}(H)$, as a sequence in $\varepsilon$, is Cauchy in $L^2(\P_{\nu_\rho})$.  

The arguments for $\K^{n,\rightarrow}_t(H)$ and $\langle \M^{n,\rightarrow}_t\rangle$ and identification of limits $\K_t$ and $\D_t(H)$, noting their forms and that the process starts from product measure $\nu_\rho$, follow straightforwardly.

  Assuming the limit with respect to $\I_t^{n, \rightarrow}(H)$, we now identify $\I_t(H)$ by approximating $\Delta^{\alpha/2}H\in C_{p,0}(\R)$ by functions $\Phi^k\in\S(\R)$ with respect to the norm on $C_{p,0}(\R)$:  Approximate, as in \eqref{approximation}, 
\begin{eqnarray*}
\E_{\nu_\rho} \Big|\int_0^t \Y_s^{n,\rightarrow}(\Delta^{\alpha/2}H)ds - \int_0^t\Y^{n,\rightarrow}_s(\Phi^k)ds\Big |^2 \ \leq \ Ct^2  \|\Delta^{\alpha/2}H - \Phi^k\|_p^2 \ \ \ {\rm and \ }\end{eqnarray*}
$$\E_{\nu_\rho}\Big|\int_0^t \Y_s(\Delta^{\alpha/2}H)ds -\int_0^t \Y_s(\Phi^k)ds\Big|^2 \ \leq \ Ct^2 \|\Delta^{\alpha/2}H - \Phi^k\|^2_p.$$
 Passing now to the limit as $n'\uparrow\infty$, one obtains $\I_t(H) = \tilde{g}'(\rho)\int_0^t\Y_s(\Delta^{\alpha/2}H)ds$.

We now argue the limit for $\B_t^{n,\rightarrow}$ assumed earlier, and remark the limit for $\I^{n,\rightarrow}_t$ is analogous since $\Delta^{\alpha/2}G\in C_{p,0}(\R)$ is uniformly continuous.  Note, for $\ell=\varepsilon n^{1/\alpha}$, to move the shift by $n^{-1/\alpha}\lfloor a\tilde{g}'(\rho)sn/n^\gamma\rfloor$ in $\nabla^n_x H_{\gamma,s}$ (cf. \eqref{shifted_H}) to $\tau_x\Y^{n, \rightarrow}_s(\iota_\varepsilon)$, write
\begin{eqnarray*}
&&\sum_{x\in \Z}(\nabla^n_x H_{\gamma,s})\Big(\big(\eta^n_s\big)^{(\ell)}(x) - \rho\Big)^2\\
&&\ \ \ \ \ \ \ = \ 
 \sum_{x\in \Z}(\nabla^n_x H_{\gamma,s})\Big(\frac{1}{2n^{1/\alpha}\varepsilon +1}\sum_{|z|\leq n^{1/\alpha}\varepsilon}(\eta^n_s(z+x) - \rho)\Big)^2\\
&&\ \ \ \ \ \ \  = \ \frac{1+O(n^{-1/\alpha})}{n^{1/\alpha}}\sum_{x\in \Z} (\nabla^n_x H)\big[\tau_x \Y^{n, \rightarrow}_s(\iota_\varepsilon)\big]^2.
\end{eqnarray*}

Then, with $\ell = \varepsilon n^{1/\alpha}$, since $\gamma = 1-3/2\alpha$, by Proposition \ref{gbg_L2}, 
we have
\begin{eqnarray*}
&&\lim_{n\uparrow\infty}\E_{\nu_\rho}\Big[\Big(\B^{n, \rightarrow}_t(H) - \beta\tilde{g}''(\rho)\int_0^t \frac{1}{n^{1/\alpha}} \sum_{x\in \Z} (\nabla^n_x H)\tau_x\Y^{n, \rightarrow}_s(\iota_\varepsilon)^2 ds\Big)^2\Big]\\
&&\  = \ \lim_{n\uparrow\infty}\E_{\nu_\rho}\Big[\Big( \B_t^{n, \rightarrow}(H) \\
&&\ \ \ \ \ \ \ \ \ \ \ - \beta\tilde{g}''(\rho)\int_0^t \frac{1}{n^{1/\alpha}} \sum_{x\in \Z} (\nabla^n_x H)\tau_x\Big\{\Y^{n, \rightarrow}_s(\iota_\varepsilon)^2 - \frac{\sigma^2_\ell(\rho)}{2\varepsilon}\Big\}ds\Big)^2\Big]\\
&&\  \leq \ \lim_{n\uparrow\infty}C(\beta,\rho,g,T)\Big( \varepsilon^{\alpha-1} + \frac{1}{\varepsilon^{3} n^{1/\alpha}}\Big)\\
&&\ \ \ \ \ \ \ \ \ \ \ \ \ \ \ \times \Big[\Big(\frac{1}{n^{1/\alpha}}\sum_{x\in \Z} \big(\nabla^n_x H\big)^2\Big) + \Big(\frac{1}{n^{1/\alpha}}\sum_{x\in \Z}\big|\nabla^n_x H\big|\Big)^2\Big].
\end{eqnarray*}
Here, as the sum of $\nabla^n_x H_{\gamma,s}$ on $x$ vanishes, the centering constant $(2\varepsilon)^{-1}\sigma^2_\ell(\rho)$ was put in the second line.  

Now, 
$$\Y^{n, \rightarrow}_s(\iota_\varepsilon)^2 - \Y^{n, \rightarrow}_x(G_\varepsilon)^2 \ =\  \big[\Y^{n, \rightarrow}_s(\iota_\varepsilon) - \Y^{n, \rightarrow}_s(G_\varepsilon)\big] \cdot \big[\Y^{n, \rightarrow}_s(\iota_\varepsilon) + \Y^{n, \rightarrow}_s(G_\varepsilon)\big],$$
 and, by Schwarz inequality, 
\begin{eqnarray*}
&&\lim_{n\uparrow\infty} \E_{\nu_\rho} \Big[\Big(\int_0^t\frac{1}{n^{1/\alpha}} \sum_{x\in \Z} (\nabla^n_x H)\tau_x\Y^{n, \rightarrow}_s(\iota_\varepsilon)^2 ds - \A^{n,\gamma,\varepsilon}_{0,t}(H)\Big)^2\Big]\\
&&\ \leq \ C(\rho)\varepsilon^{-1}\|G_\varepsilon-\iota_\varepsilon\|^2_{L^2(\R)}t^2\Big(\frac{1}{n^{1/\alpha}}\sum_{x\in \Z}\big|\nabla^n_x H\big|\Big)^2.
\end{eqnarray*}
These estimates with the inequality $(a+b)^2 \leq 2a^2 + 2b^2$ finish the proof of the $\B^{n, \rightarrow}_t$ limit.

We now address the martingale convergence. 
By the identification given before, any limit point of the quadratic variation sequence equals $\mc D_t(H) = \tilde{g}(\rho)t\|\nabla^{\alpha/2} H\|^2_{L^2(\R\times \R)}$.  Then, the limit of martingales $\M_t(H)$ with respect to the uniform topology is a continuous martingale.  Also, by the triangle inequality, Doob's inequality and the quadratic variation bound \eqref{quad_var_bound}, 
\begin{eqnarray*}
&&\sup_n \E_{\nu_\rho}\Big[\sup_{0\leq s\leq t} |\M^{n, \rightarrow}_s(H) - \M^{n, \rightarrow}_{s^-}(H)|\Big]
\ \leq \ 2\sup_n \E_{\nu_\rho} \Big[ \sup_{u\in [0,t]}|\M^{n, \rightarrow}_u(H)|^2\Big]^{1/2}\\
&& \ \ \ \ \ \ \ \ \ \ \ \ \ \ \  \leq \ 2\sup_n \E_{\nu_\rho}\Big[\langle M^{n, \rightarrow}_t(H)\rangle\Big]^{1/2} \ \leq \ C(\beta,T)\|g\|_{L^1(\nu_\rho)}\|\nabla^{\alpha/2} H\|^2_{L^2(\R\times \R)}.
\end{eqnarray*}
Then, by Corollary VI.6.30 of
\cite{JS}, $(\M_t^{n', \rightarrow}(H), \langle \M^{n', \rightarrow}_t(H)\rangle)$ converges subsequentially in distribution to $(\M_t(H), \langle \M_t(H)\rangle)$.  Since, also $\langle \M^{n', \rightarrow}_t(H)\rangle$ converges on a subsequence in distribution to $\D_t(H)=\tilde{g}(\rho)t\|\nabla^{\alpha/2} H\|^2_{L^2(\R\times \R)}$, we have $\langle \M_t(H)\rangle = \tilde{g}(\rho)t\|\nabla^{\alpha/2} H\|^2_{L^2(\R\times \R)}$.
Hence, by Levy's theorem, $\M_t$ is a version of the noise $N_t$ desired.
This finishes the proof of (1) and (3).

 To show (2), by a Fatou's lemma,
one need only show $\lim_{n\uparrow\infty}\B^{n,\gamma}_t(H)=0$ in $L^2(\P_{\nu_\rho})$.    When $0<\alpha<3/2$, by Proposition \ref{gbg_L2} and the argument for part (1) with $\varepsilon =1$, $n^{\gamma  -1 + 3/2\alpha}\B^{n,\rightarrow}_t(H)$ differs from $\beta \tilde{g}''(\rho)\A^{n,1,\rightarrow}_{0,t}(H)$ by an error of $O(1)$.  Since $\A^{n,1,\rightarrow}_{0,t}(H)$ is bounded in $L^2(\nu_\rho)$, $\B^{n,\rightarrow}_t(H)$ must vanish in $L^2(\nu_\rho)$.
\end{proof}

\begin{proposition}
\label{limit_eqn}
Consider the processes in Theorem \ref{secondorder_thm} when $0<\alpha<3/2$.  We have that limit points $\{\Y_t: t\in [0,T]\}$ satisfy equation \eqref{3.3} for $\Phi_t\in \S(\R)\otimes \widehat{C}$.
\end{proposition}

\begin{proof}
We show that \eqref{3.3} holds for a function of the form $\Phi_t(x) = \Phi(x)f(t)$ for $\Phi\in \S(\R)$ and $f\in \widehat{C}$.  Then, noting the covariance $\E_{\nu_\rho} [\Y_t(G)\Y_t(H)] = \sigma^2(\rho)\int_\R GH dx$ and $\E_{\nu_\rho}[N_t(G)N_t(H)] = \sigma^2(\rho)\int_\R G\Delta^{\alpha/2}H dx$, the  usual approximation procedure can be employed to verify \eqref{3.3} for $\Phi_t\in \S(\R)\otimes \widehat{C}$.

Multiplying the decomposition \eqref{mart_decomposition} by $f'(t)$ and then integrating over $t\in [0,T]$, we obtain
\begin{eqnarray*}
\int_0^Tf'(t)\Y^{n, \rightarrow}_t(\Phi) dt 
& = & \int_0^T f'(t)dt \Y^{n, \rightarrow}_0(\Phi) + \int_0^T f'(t)\int_0^t \Y^{n, \rightarrow}_s(\Delta^{\alpha/2}\Phi)dsdt \\
&&\ \ \ \ \ \ + \int_0^T f'(t)\M^{n, \rightarrow}_t (\Phi)dt + \mathcal{E}(n)\end{eqnarray*}
where $\mathcal{E}(n)$ incorporates $\B^{n,\rightarrow}_t(\Phi)$ and other errors.
  By part (2) of Proposition \ref{stat_lemma1}, we conclude $\lim_{n\uparrow\infty}\E_{\nu_\rho}[\mathcal{E}^2(n)] =0$.

Since $f(T)=0$, $\int_0^T f'(t)dt \Y^{n, \rightarrow}_0(\Phi) = -f(0)\Y^{n, \rightarrow}_0(\Phi)$ and 
$$\int_0^T f'(t)\int_0^t\Y^{n, \rightarrow}_s(\Delta^{\alpha/2}\Phi)dsdt \ = \ -\int_0^T f(t)\Y^{n, \rightarrow}_t(\Delta^{\alpha/2}\Phi)dt.$$
Then, passing along the subsequence $n=n'$, and approximating $\Delta^{\alpha/2}\Phi$ by functions in $\S(\R)$ in the $\|\cdot\|_p$ norm, we obtain
$$\int_0^Tf'(t)\Y_t(\Phi)dt \ = \ -f(0)\Y_0(\Phi) - \int_0^T f(t)\Y_t(\Delta^{\alpha/2}\Phi)dt + \int_0^T f'(t) \M_t(\Phi_t)dt.$$
Recall $\M_t$ is a version of the noise $N_t$.  \end{proof}

\noindent
 {\bf Proof of Theorem \ref{secondorder_thm}}  Let $H\in \S(\R)$ and $t\in [0,T]$.  When $3/2\leq \alpha<2$ and $\gamma = 1-3/2\alpha$, by the decomposition \eqref{mart_decomposition}, Proposition \ref{stat_lemma1}, and tightness of the constituent processes $\Y_t^{n, \rightarrow}$, $\M^{n, \rightarrow}_t$, $\I^{n, \rightarrow}_t$, $\B^{n, \rightarrow}_t$, $\K^{n, \rightarrow}_t$ and $\langle \M^{n, \rightarrow}_t\rangle$ in the uniform topology (Proposition \ref{stationary_tightness}), any limit point of 
$$\left (\Y^{n, \rightarrow}_t, \M^{n, \rightarrow}_t, \I^{n, \rightarrow}_t, \B^{n, \rightarrow}_t, \K^{n, \rightarrow}_t, \langle \M^{n, \rightarrow}_t\rangle : t\in [0,T] \right )$$
 satisfies
$$\M_t(H) \ =\ \Y_t(H) - \Y_0(H) - \tilde{g}'(\rho)\int_0^t\Y_s(\Delta^{\alpha/2} H)ds -\beta\tilde{g}''(\rho)\A_{0,t}(H).$$
 Moreover, by Proposition \ref{stat_lemma1}, $\Y_t$ is a fractional $L^2$ energy solution of \eqref{gen_OU_sec}.

Finally, when $0<\alpha<3/2$, by Proposition \ref{limit_eqn}, all limit points $\Y_\cdot$ satisfy equation \eqref{3.3}.  Hence, by uniqueness of the `generalized' solution to this equation, all limit point converge to this solution which solves \eqref{gen_OU_DG}.
\qed

\medskip

\noindent{\bf Acknowledgments.}
We thank Tadahisa Funaki and Jeremy Quastel for discussions on the fractional KPZ Burgers equation \eqref{fractional KPZ Burgers}.

This work was partially supported by ARO grant 65389MA.

\end{document}